\documentclass[12pt]{amsart}

\usepackage{amssymb}
\usepackage{bm}
\usepackage{graphicx}
\usepackage{psfrag}
\usepackage[usenames,dvipsnames]{xcolor}
\usepackage{subcaption}
\usepackage{enumerate}
\usepackage{url}

\usepackage{algpseudocode}

\usepackage{rotating}

\usepackage{mathdots}
\usepackage{mathtools}

\usepackage{xy}
\input xy
\xyoption{all}

\usepackage{tikz}
\usetikzlibrary{positioning} 
\usetikzlibrary{graphs,graphs.standard}

\newcommand{\excise}[1]{}

\newtheorem{thm}{Theorem}[section]
\newtheorem{lemma}[thm]{Lemma}

\newtheorem{cor}[thm]{Corollary}
\newtheorem{prop}[thm]{Proposition}

\theoremstyle{definition}

\newtheorem{example}[thm]{Example}
\newtheorem{remark}[thm]{Remark}
\newtheorem{defn}[thm]{Definition}

\newtheorem{notation}[thm]{Notation}

\numberwithin{equation}{section}

\newcommand{\ring}[1]{\ensuremath{\mathbb{#1}}}

\renewcommand\>{\rangle}
\newcommand\<{\langle}

\newcommand\RR{\ring{R}}

\newcommand\ZZ{\ring{Z}}

\newcommand\iso{\cong}

\def\ol#1{{\overline {#1}}}

\DeclareMathOperator\Ap{Ap} 

\begin{document}

\mbox{}
\title[Numerical semigroups, polyhedra, and posets II]{Numerical semigroups, polyhedra, and posets II:\ \\ locating certain families of semigroups}

\author[J.~Autry]{Jackson Autry}
\address{Mathematics and Statistics Department\\San Diego State University\\San Diego, CA 92182}
\email{jautry@sdsu.edu}

\author[A.~Ezell]{Abigail Ezell}
\address{Mathematics and Computer Science Department\\Colorado College\\Colorado Springs, CO 80903}
\email{a\_ezell@coloradocollege.edu}

\author[T.~Gomes]{Tara Gomes}
\address{Mathematics and Statistics Department\\San Diego State University\\San Diego, CA 92182}
\email{gomes.tara@gmail.com}

\author[C.~O'Neill]{Christopher O'Neill}
\address{Mathematics and Statistics Department\\San Diego State University\\San Diego, CA 92182}
\email{cdoneill@sdsu.edu}

\author[C.~Preuss]{Christopher Preuss}
\address{Mathematics Department\\University of Washington Tacoma\\Tacoma, WA 98402}
\email{preussc@uw.edu}

\author[T.~Saluja]{Tarang Saluja}
\address{Mathematics and Statistics Department\\Swarthmore College\\Swarthmore, PA 19081}
\email{tsaluja1@swarthmore.edu}

\author[E.~Davila]{Eduardo Torres Davila}
\address{Mathematics and Statistics Department\\San Diego State University\\San Diego, CA 92182}
\email{etdavila10@gmail.com}

\date{\today}

\begin{abstract}
Several recent papers have examined a rational polyhedron $P_m$ whose integer points are in bijection with the numerical semigroups (cofinite, additively closed subsets of the non-negative integers) containing $m$.  A combinatorial description of the faces of $P_m$ was recently introduced, one that can be obtained from the divisibility posets of the numerical semigroups a given face contains.  
In this paper, we study the faces of $P_m$ containing arithmetical numerical semigroups and those containing certain glued numerical semigroups, as an initial step towards better understanding the full face structure of~$P_m$.  In most cases, such faces only contain semigroups from these families, yielding a tight connection to the geometry of $P_m$.  
\end{abstract}

\maketitle


\section{Introduction}
\label{sec:intro}

A \emph{numerical semigroup} is a cofinite subset $S \subseteq \ZZ_{\ge 0}$ of the non-negative integers that is closed under addition.  Numerical semigroups are often specified using a set of generators $n_1 < \cdots < n_k$, i.e., 
$$S = \<n_1, \ldots, n_k\> = \{a_1n_1 + \cdots + a_kn_k : a_i \in \ZZ_{\ge 0}\}.$$
The \emph{Ap\'ery set} of $m \in S$ is the set 
$$\Ap(S;m) = \{n \in S : n - m \notin S\}$$
of the minimal elements of $S$ within each equivalence class modulo $m$.  Since $S$ is cofinite, we~are guaranteed $|\!\Ap(S;m)| = m$, and that $\Ap(S;m)$ contains exactly one element in each equivalence class modulo $m$.  The elements of $\Ap(S;m)$ are partially ordered by divisibility, that is, $a \preceq a'$ whenever $a' - a \in S$ (or, equivalently, whenever $a' - a \in \Ap(S;m)$); we call this the \emph{Ap\'ery poset} of $m$ in $S$.   

A family of rational polyhedra whose integer points are in bijection with certain numerical semigroups, first introduced by Kunz~\cite{kunz} and independently in~\cite{kunzcoords}, has received a flurry of recent attention~\cite{alhajjarkunz,wilfmultiplicity,oversemigroupcone,kaplanwilfconj,kunzfaces1,kunznew,kunzcoords}.  
More specifically, given $m \ge 2$, the \emph{Kunz polyhedron} $P_m$ is a pointed rational cone, translated from the origin, whose integer points are in bijection with the numerical semigroups containing $m$ (we defer precise definitions to Section~\ref{sec:polyhedra}).  One of the primary goals of studying these polyhedra is to utilize tools from lattice point geometry (e.g., Ehrhart's theorem) to approach some long-standing enumerative questions involving numerical semigroups~\cite{wilfsurvey,kaplanwilfconj}.  

Recent developments have produced a combinatorial description of the faces of $P_m$.  Given a numerical semigroup $S$ containing $m$, the \emph{Kunz poset} is the partially ordered set with ground set $\ZZ_m$ obtained by replacing each element of the Ap\'ery poset $\Ap(S;m)$ with its equivalence class in $\ZZ_m$.  In~\cite{wilfmultiplicity}, it was shown that two numerical semigroups lie in the interior of the same face of~$P_m$ if and only if they have identical Kunz posets, thereby providing a natural combinatorial object that indexes the faces of~$P_m$ containing numerical semigroups.  This combinatorial description was extended in~\cite{kunzfaces1} to include every face of $P_m$, even those that do not contain any numerical semigroups.  

When studying questions that are difficult to answer for general numerical semigroups, it is common to restrict to certain families with some additional structure.  In~this paper, we examine two such families.  The first (one of the most common in the literature) are \emph{arithmetical} numerical semigroups, whose minimal generating sets are arithmetic sequences.  Thanks to a particularly powerful membership criterion, these semigroups admit closed forms for many quantities that are difficult to obtain in general~\cite{setoflengthsets,omidali,numerical}.  The second family is comprised of numerical semigroups obtained by scaling every element of a given semigroup $S$ by some common factor $\beta$ and then adding one new generator $\alpha$ to obtain $\<\alpha\> + \beta S$ (this process is known as \emph{gluing}).  The~result is a broad class of semigroups (which we call \emph{monoscopic} numerical semigroups (Definition~\ref{d:monoscopic})) that includes several other families of independent interest, such as supersymmetric~\cite{supersymmetric} and telescopic~\cite{telescopic} numerical semigroups, as well as numerical semigroups on compound sequences~\cite{compseqs}.  

The combinatorial interpretation of the faces of $P_m$ in terms of posets is still young, and many basic questions are still unanswered.  
The goal of this paper is to describe geometrically the faces of~$P_m$ containing numerical semigroups from the families described above, as an initial step towards better understanding the full face structure of~$P_m$.  To this end, we give a formula for the dimension of every such face, and in most cases characterize their extremal rays (both of which are still not well understood in general for~$P_m$).  Our~results yield two particularly notable geometric insights.  

\begin{itemize}
\item 
We provide a collection of combinatorial embeddings of the form $P_m \hookrightarrow P_{\beta m}$, which we call \emph{monoscopic embeddings}, whose images contain precisely those semigroups obtained from monoscopic gluings of semigroups in $P_m$ with scaling factor~$\beta$.  This provides a complete characterization of the faces containing monoscopic numerical semigroups, as well as all faces they contain.  One interpretation of this construction is that monoscopic gluing of numerical semigroups can be realized as a geometric operation on Kunz polyhedra.  

\item 
The posets associated to low-dimensional faces, such as rays, possess the most relations, making them more difficult to classify in general.  Moreover, many rays do not contain numerical semigroups.  When describing a ray~$\vec r$ of a face $F$ containing arithmetical numerical semigroups, we do so by examining the effect adding $\vec r$ to each integer point in $F$ has on the minimal generators of the corresponding semigroup.  

\end{itemize}

In addition to providing a glimpse of the face structure of $P_m$, our results have several consequences for numerical semigroups outside the realm of geometry.  

\begin{itemize}
\item 
The elements in the Ap\'ery sets of arithmetical and monoscopic numerical semigroups are well understood (indeed, this is one of the reasons these families are considered especially ``nice'').  We extend both of these classical results to include a description of the divisibility poset structure of the Ap\'ery set, each of which has an elegant combinatorial structure; see Figure~\ref{f:previews} for examples.  

\begin{figure}[t!]
\begin{center}
\begin{subfigure}[t]{0.45\textwidth}
\begin{center}
\includegraphics[height=2.5in]{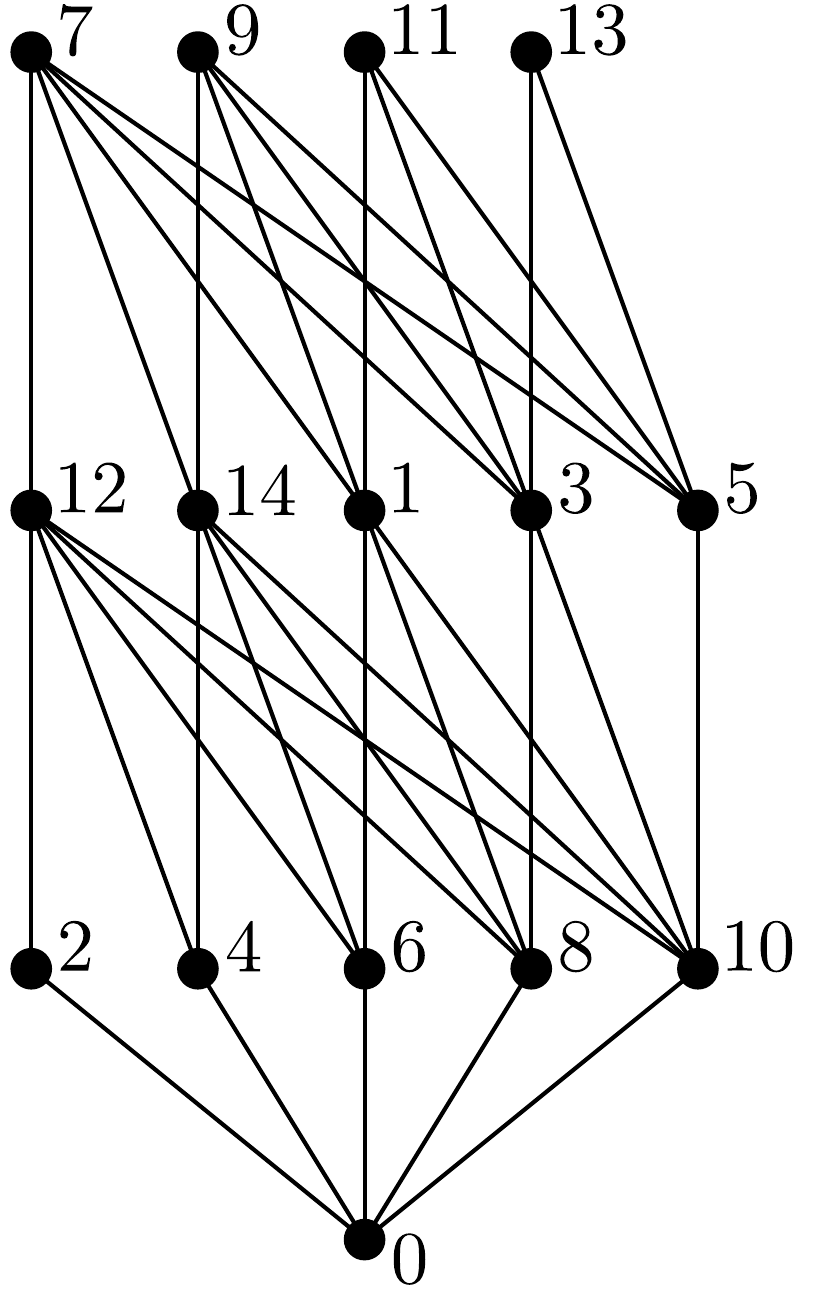}
\end{center}
\caption{}
\label{f:previewarithmetic}
\end{subfigure}
\hspace{0.02\textwidth}
\begin{subfigure}[t]{0.45\textwidth}
\begin{center}
\includegraphics[height=2.5in]{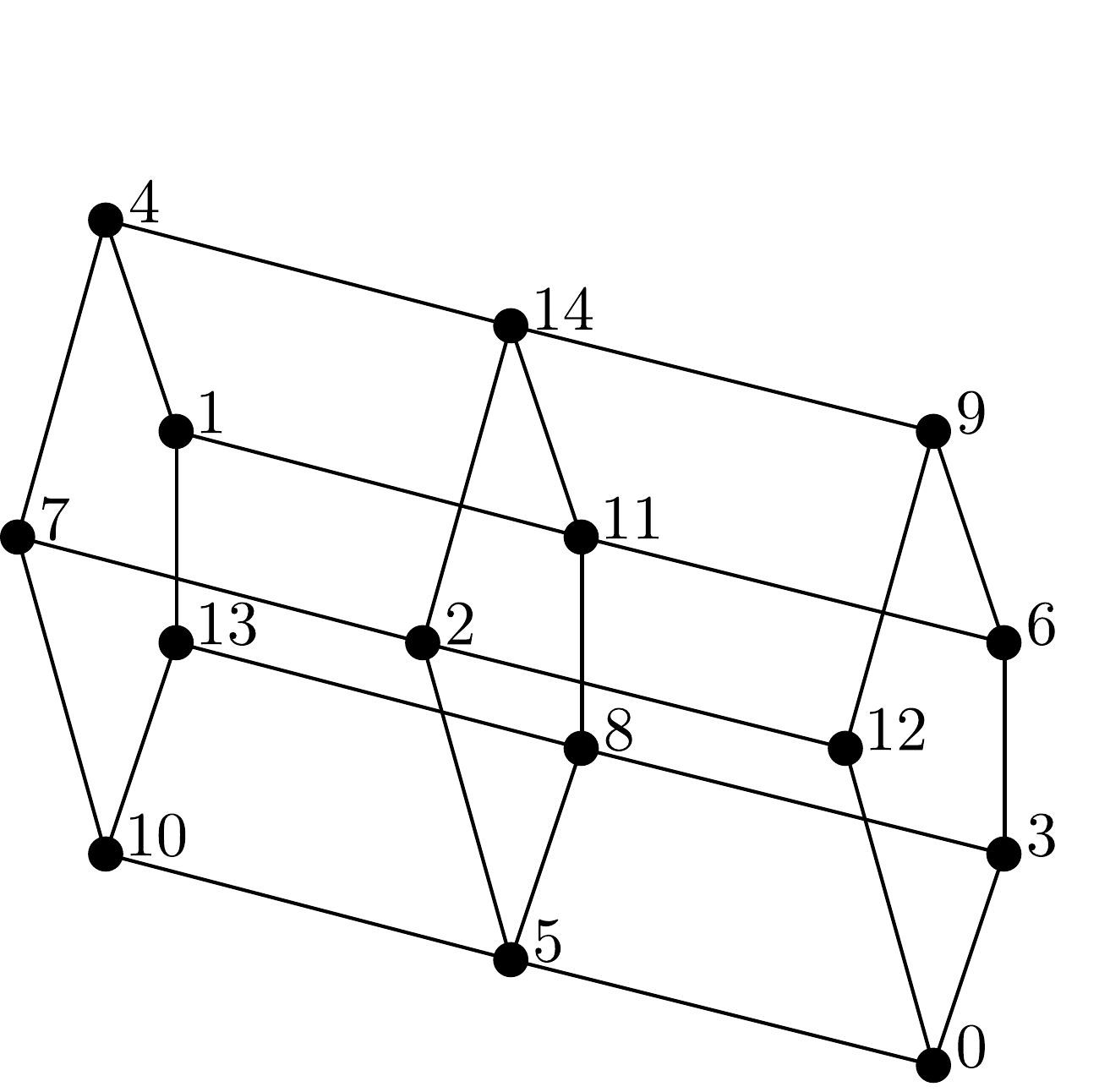}
\end{center}
\caption{}
\label{f:previewscopic}
\end{subfigure}
\end{center}
\caption{Ap\'ery posets of two numerical semigroups.  The first, given by $S = \<15,17,19,21,23\>$ (left), is arithmetical, and the second, given by $S' = \<15, 18, 20, 27\>$ (right), is monoscopic.}
\label{f:previews}
\end{figure}

\item 
In most cases, the membership criterion for generalized arithmetical numerical semigroups can be extended to all semigroups lying on their same face.  This gives rise to a new family of semigroups, which we call \emph{extra-generalized arithmetical numerical semigroups}, possessing most of the desirable properties of arithmetical numerical semigroups.  We develop this new family, independent of the geometry of $P_m$, including a membership criterion (Proposition~\ref{p:arithmembercrit}) and a formula for their Frobenius number (Corollary~\ref{c:extraarithfrob}).  

\end{itemize}

The paper is organized as follows.  After reviewing the necessary terminology in Section~\ref{sec:polyhedra}, we introduce extra-generalized arithmetical numerical semigroups in Section~\ref{sec:arithposets}, providing a membership criterion (Proposition~\ref{p:arithmembercrit}), a characterization of their Ap\'ery posets (Theorem~\ref{t:arithposet}), and a formula for their Frobenius numbers (Corollary~\ref{c:extraarithfrob}).  We then examine the faces of $P_m$ containing extra-generalized arithmetical numerical semigroups in Section~\ref{sec:arithfaces}, giving a formula for their dimension (Theorem~\ref{t:arithfacedim}) and, in most cases, their extremal rays (Theorem~\ref{t:arithrays}).  In the final two sections of the paper, we turn attention to monoscopic numerical semigroups, characterizing their Ap\'ery poset structure (Theorem~\ref{t:monoscopicposet}) and the complete structure of the faces containing them (Theorems~\ref{t:augmonoscopicfaces} and~\ref{t:monoscopicfaces}) via monoscopic embeddings (Definition~\ref{d:monoscopicembedding}).

\section{The group cone and its faces}
\label{sec:polyhedra}

After recalling basic definitions from polyhedral geometry (see~\cite{ziegler} for a thorough introduction), we define the Kunz polyhedron $P_m$ and a related polyhedron 
from~\cite{kunzfaces1}.  

A \emph{rational polyhedron} $P \subset \RR^d$ is the set of solutions to a finite list of linear inequalities with rational coefficients, that is, 
$$P = \{x \in \RR^d : Ax \le b\}$$
for some matrix $A$ and vector $b$.  
If none of the inequalities can be omitted without altering $P$, we call this list the \emph{$H$-description} or \emph{facet description} of $P$ (such a list of inequalities is unique up to reordering and scaling by positive constants).  The~inequalities appearing in the H-description of $P$ are called \emph{facet inequalities} of~$P$.  

Given a facet inequality $a_1x_1 + \cdots + a_dx_d \le b$ of $P$, the intersection of $P$ with the equation $a_1x_1 + \cdots + a_dx_d = b$ is called a \emph{facet} of $P$.  
A \emph{face} $F$ of $P$ is a subset of $P$ equal to the intersection of some collection of facets of $P$.  The set of facets containing $F$ is called the \emph{H-description} or \emph{facet description} of $F$.  The \emph{dimension} of a face $F$ is the dimension $\dim(F)$ of the affine linear span of $F$.  The \emph{relative interior} of a face $F$ is the set of points in $F$ that do not also lie in a face of dimension strictly smaller than $F$ (or, equivalently, do not lie in a proper face of $F$).  We say $F$ is a \emph{vertex} if $\dim(F) = 0$, and \emph{edge} if $\dim(F) = 1$ and $F$ is bounded, a \emph{ray} if $\dim(F) = 1$ and $F$ is unbounded, and a \emph{ridge} if $\dim(F) = d - 2$.  

If there is a unique point $v$ satisfying every inequality in the H-description of $P$ with equality, then we call $P$ a \emph{cone} with vertex $v$.  If, additionally, $b = 0$ above, we call $P$ a \emph{pointed cone}.  Separately, we say $P$ is a \emph{polytope} if $P$ is bounded.  If $P$ is a pointed cone, then any face $F$ equals the non-negative span of the rays of $P$ it contains, and if~$P$ is a polytope, then any face $F$ equals the convex hull of the set of vertices of $P$ it contains; in each case, we call this the \emph{V-description} of $F$.  

A \emph{partially ordered set} (or \emph{poset}) is a set $Q$ equipped with a \emph{partial order} $\preceq$ that is reflexive, antisymmetric, and transitive.  We say $q$ \emph{covers} $q'$ if $q' \prec q$ and and there is no intermediate element $q''$ with $q' \prec q'' \prec q$.  If $(Q, \preceq)$ has a unique minimal element $0 \in Q$, the \emph{atoms} of $Q$ are the elements that cover $0$.  The set of faces of a polyhedron $P$ forms a poset under containment that is a \emph{lattice} (i.e., every element has a unique greatest common divisor and least common multiple) and is \emph{graded}, where the height function is given by dimension.  If $P$ is a cone, then every face of $P$ equals the sum of some collection of extremal rays and the intersection of some collection of facets, meaning the face lattice of $P$ is both \emph{atomic} and \emph{coatomic}.  

\begin{defn}\label{d:kunzcoords}
Fix $m \in \ZZ_{\ge 2}$, and a numerical semigroup $S$ containing $m$.  Write
$$\Ap(S;m) = \{0, a_1, \ldots, a_{m-1}\},$$
where $a_i = mz_i + i$ for each $i = 1, \ldots, m-1$.  We refer to the tuples $(a_1, \ldots, a_{m-1})$ and $(z_1, \ldots, z_{m-1})$ as the \emph{Ap\'ery tuple}/\emph{Ap\'ery coordinates} and the \emph{Kunz tuple}/\emph{Kunz coordinates} of $S$, respectively.  
\end{defn}

\begin{defn}\label{d:kunzandgroupcone}
Fix a finite Abelian group $G$, and let $m = |G|$.  The \emph{group cone} $\mathcal C(G) \subset \RR^{m-1}$ is the pointed cone with facet inequalities
$$\begin{array}{r@{}c@{}l@{\qquad}l}
x_i + x_j &{}\ge{}& x_{i+j} & \text{for } i, j \in G \setminus \{0\} \text{ with } i + j \ne 0,
\end{array}$$
where the coordinates of $\RR^{m-1}$ are indexed by $G \setminus \{0\}$.  
Additionally, for each integer $m \ge 2$, let $P_m$ denote the translation of $\mathcal C(\ZZ_m)$ with vertex $(-\tfrac{1}{m}, \ldots, -\tfrac{m-1}{m})$, whose facets are given by
$$\begin{array}{r@{}c@{}l@{\qquad}l}
x_i + x_j &{}\ge{}& x_{i+j} & \text{for } 1 \le i \le j \le m - 1 \text{ with } i + j < m, \text{ and} \\
x_i + x_j + 1 &{}\ge{}& x_{i+j-m} & \text{for } 1 \le i \le j \le m - 1 \text{ with } i + j > m.  
\end{array}$$
We refer to $P_m$ as the \emph{Kunz polyhedron}.  
\end{defn}

Parts~(a) and~(b) of the following theorem appear in \cite{kunz} and \cite{kunzfaces1}, respectively.  

\begin{thm}\label{t:kunzlatticepts}
Fix an integer $m \ge 2$.  
\begin{enumerate}[(a)]
\item 
The set of all Kunz tuples of numerical semigroups containing $m$ coincides with the set of integer points in $P_m$.  

\item 
The set of all Ap\'ery tuples of numerical semigroups containing $m$ coincides with the set of integer points $(a_1, \ldots, a_{m-1})$ in $\mathcal C(\ZZ_m)$ with $a_i \equiv i \bmod m$ for every $i$.  

\end{enumerate}
\end{thm}

In view of Theorem~\ref{t:kunzlatticepts}, given a face $F \subset \mathcal C(\ZZ_m)$, we say $F$ \emph{contains} a numerical semigroup $S$ if the Ap\'ery tuple lies in the relative interior of $F$.  Analogously, we say a face $F' \subset P_m$ \emph{contains} $S$ if the Kunz tuple of $S$ lies in the relative interior of $F'$.

\begin{thm}[{\cite[Theorem~3.3]{kunzfaces1}}]\label{t:groupconefacelattice}
Fix a finite Abelian group $G$ and a face $F \subset \mathcal C(G)$.  
\begin{enumerate}[(a)]
\item 
The set
$H = \{h \in G : x_h = 0 \text{ for all } x \in F\}$
is a subgroup of $G$ (called the \emph{Kunz subgroup} of $F$), and the relation $P = (G/H, \preceq)$ with unique minimal element $\ol 0$ and $\ol a \preceq_P \ol b$ whenever $x_a + x_{b-a} = x_b$ for distinct $a, b \in G$ is a well defined partial order (called the \emph{Kunz poset} of $F$).  

\item 
If $G = \ZZ_m$ with $m \ge 2$ and $F$ contains a numerical semigroup $S$, then the Kunz subgroup of $F$ is trivial and the Kunz poset of $F$ equals the Kunz poset of $S$.  

\item 
In the Kunz poset $P$ of $F$, $\ol b$ covers $\ol a$ if and only if $\ol b - \ol a$ is an atom of $P$.  

\end{enumerate}
\end{thm}

\begin{remark}\label{r:automorphisms}
The automorphism group of a finite Abelian group $G$ acts on the face lattice of $\mathcal C(G)$ by permuting coordinates in the natural way.  Each such automorphsim induces a group isomorphism between the corresponding Kunz subgroups, as well as on the corresponding Kunz posets.  
\end{remark}

\section{Extra-generalized arithmetical numerical semigroups}
\label{sec:arithposets}

Arithmetical numerical semigroups, which are numerical semigroups whose minimal generating set is an arithmetic sequence, are a common focal point in the literature.  It~turns out the polyhedral faces that contain arithmetical numerical semigroups also contain semigroups from two related families:\ generalized arithmetical numerical semigroups (previously studied in \cite{omidali,omidalirahmati}) and a new family (Definition~\ref{d:extragenarith}).  In this section, we provide a characterization of the Kunz posets for this new family (Corollary~\ref{c:arithposet}) using an adapted membership criterion (Proposition~\ref{p:arithmembercrit}), as well as a formula for their Frobenius numbers (Corollary~\ref{c:extraarithfrob}).  

\begin{defn}\label{d:extragenarith}
An \emph{extra-generalized arithmetical numerical semigroup} has the form
$$S = \<a, ah + d, ah + 2d, \ldots, ah + kd\>$$
for $a, h, k \in \ZZ_{\ge 1}$ and $d \in \ZZ$ with $k < a$, $\gcd(a,d) = 1$ and $ah + kd > a$.  If $d \ge 1$, then we call $S$ a \emph{generalized arithmetical numerical semigroup}, and if $d < 0$, then we call $S$ a \emph{pessimistic arithmetical numerical semigroup}.  
\end{defn}

\begin{example}\label{e:extragenarith}
Consider the semigroup $S = \<11,12,14,16,18,20\>$, whose Kunz poset is depicted in Figure~\ref{f:pessimisticposet}.  Reordering the generators as $S = \<11,20,18,16,14,12\>$ reveals that $S$ is extra-generalized arithmetical with $d = -2$.  This has the same Kunz poset as $S' = \<11,20,29,38,47,56\>$, whose common difference $d' = 9$ satisfies $d' \equiv d \bmod a$.  

It is important to note the extra requirement that $ah + kd > a$ in Definition~\ref{d:extragenarith} is in place to ensure the given generating set is minimal.  Indeed, the semigroup $\<11,20,18,16,14,12,10\>$ demonstrates this need not be the case if the assumption is dropped.  As it turns out, this assumption also forces $a$ to be the multiplicity of~$S$.  
\end{example}

\begin{figure}[t!]
\begin{center}
\begin{subfigure}[t]{0.30\textwidth}
\begin{center}
\includegraphics[height=1.5in]{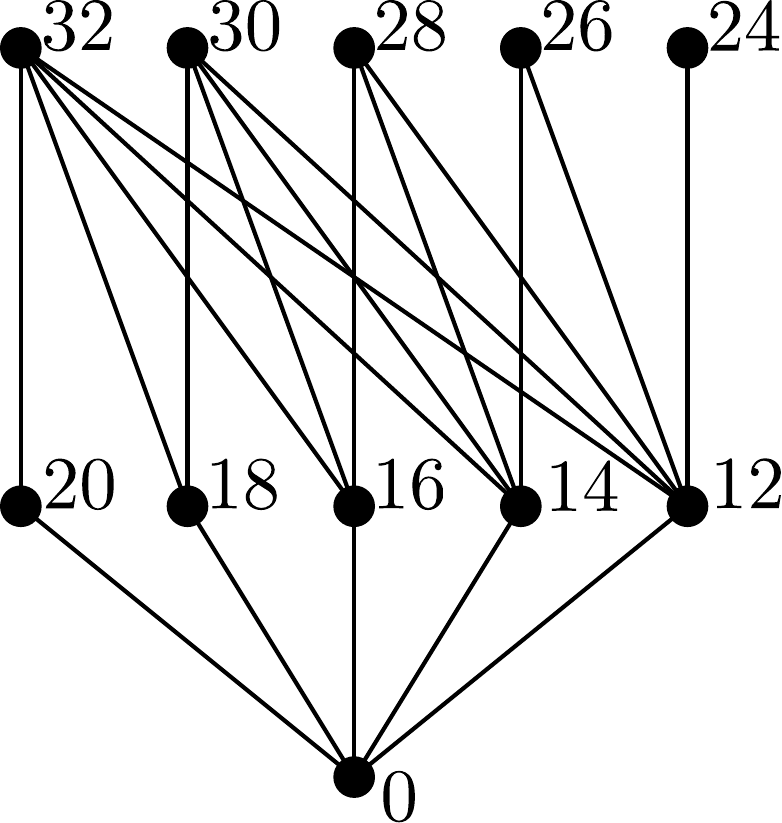}
\end{center}
\caption{}
\label{f:pessimisticposet}
\end{subfigure}
\hspace{0.02\textwidth}
\begin{subfigure}[t]{0.30\textwidth}
\begin{center}
\includegraphics[height=1.5in]{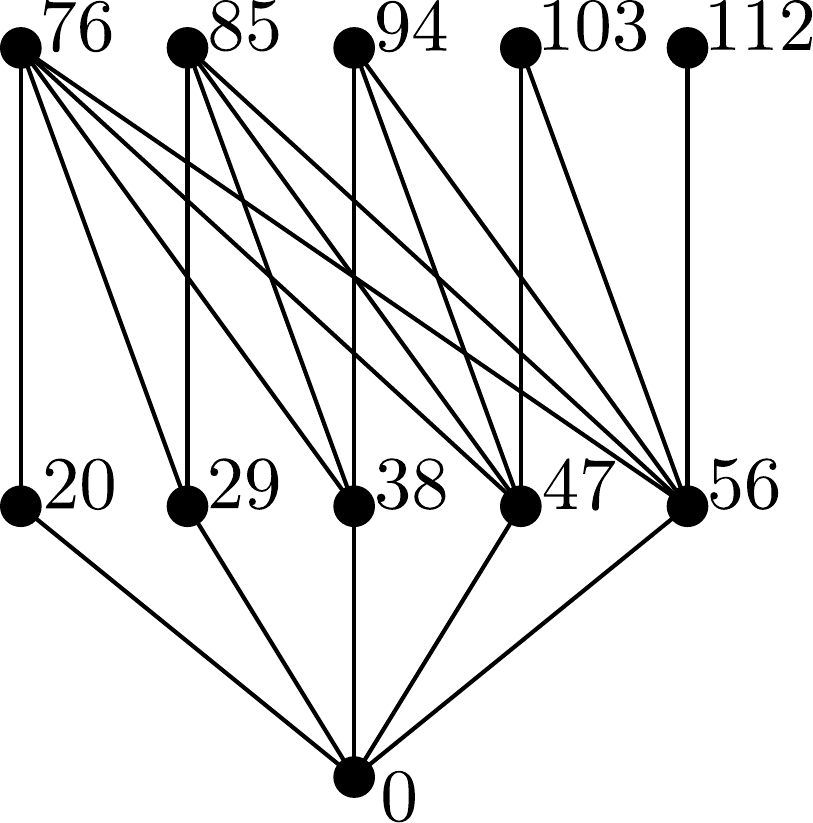}
\end{center}
\caption{}
\label{f:nonpessimisticposet}
\end{subfigure}
\hspace{0.02\textwidth}
\begin{subfigure}[t]{0.30\textwidth}
\begin{center}
\includegraphics[height=1.5in]{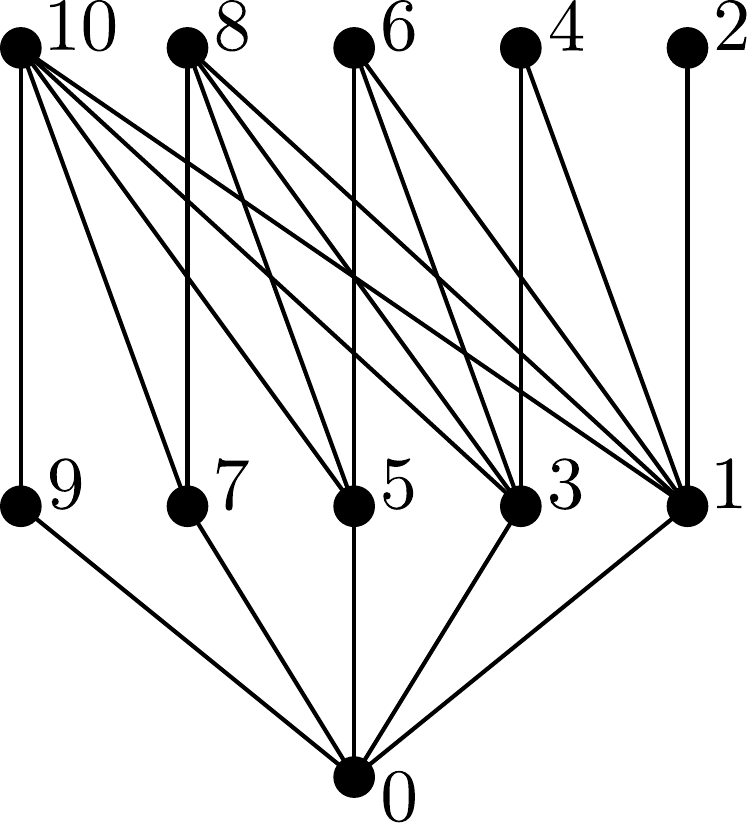}
\end{center}
\caption{}
\label{f:kunzpessimisticposet}
\end{subfigure}
\end{center}
\caption{Ap\'ery posets of the semigroups $S = \<11,12,14,16,18,20\>$ (left) and $S' = \<11,20,29,38,47,56\>$ (middle) from Example~\ref{e:extragenarith}, along with their shared Kunz poset (right).}
\end{figure}

We begin by generalizing a membership criterion for generalized arithmetical numerical semigroups to the extra-generalized family.  

\begin{prop}\label{p:arithmembercrit}
Fix an extra-generalized arithmetical numerical semigroup
$$S = \<a, ah + d, ah + 2d, \ldots, ah + kd\>.$$
Fix $n \in \ZZ$, and let $n = qa + rd$ for $q, r \in \ZZ_{\ge 0}$ with $0 \le r \le a - 1$.  We have 
\begin{enumerate}[(a)]
\item 
$n \in S$ if and only if $\lceil \tfrac{r}{k} \rceil h \le q$, and

\item 
$n \in \Ap(S;a)$ if and only if $\lceil \tfrac{r}{k} \rceil h = q$.  

\end{enumerate}
\end{prop}

\begin{proof}
First, suppose $n \in S$, so for some $z_0, z_1, \ldots, z_k \in \ZZ_{\ge 0}$, we have 
$$n = z_0 a + \sum_{i=1}^k z_i(ah + id).$$
Write $\sum_{i=1}^k z_i i = q'a + r$ for $q', r \in \ZZ_{\ge 0}$ with $r < a$.   Letting
$$q = z_0 + h\sum\limits_{i=1}^k z_i + q' d,$$
we obtain
$$
qa + rd
= z_0 a + ah\sum_{i=1}^k z_i + q'ad + rd
= z_0 a + \sum_{i=1}^k z_i (ah + id)
=n
$$
and
$$
\bigg\lceil \frac{r}{k} \bigg\rceil h
\le \bigg\lceil \frac{1}{k}\sum_{i=1}^k z_i i \bigg\rceil h
\le \bigg\lceil \sum_{i=1}^k z_i \bigg\rceil h
= h \sum_{i=1}^k z_i
\le q.
$$

Conversely, assume $\lceil \frac{r}{k} \rceil h \le q$ for some $n = aq + rd$ with $q, r \in \ZZ_{\ge 0}$ and $0 \le r \le a - 1$.  If $r = 0$, then $n = aq$ clearly implies $n \in S$, and if $0 < r \le k$, then the bounds on $r$ imply $\lceil \frac{r}{k} \rceil = 1$, so $h \le q$ and 
$$n = aq + rd = (q - h)a + (ha + rd) \in S.$$
Lastly, if $k < r$, then 
$$n = qa + rd = (q - h + h)a + (r - k + k)d = (ha + kd) + (q - h)a + (r - k)d,$$
so it suffices to show $(q - h)a + (r - k)d \in S$.  Since $0 \le r - k \le a - 1$ and
$$\bigg\lceil \frac{r-k}{k} \bigg\rceil h = \bigg\lceil \frac{r}{k} \bigg\rceil h - h \le q - h,$$
we conclude by induction on $r$ that $n \in S$.  This completes the proof of part~(a).  

For part~(b), $n \in \Ap(S;a)$ occurs when $n \in S$ and $n - a \notin S$.  Writing $n = qa + rd$ as above, we check by part~(a) that this happens if and only if $q - 1 < \lceil \frac{r}{k} \rceil h \le q$, which is equivalent to the desired equality.  
\end{proof}

\begin{thm}\label{t:arithposet}
Fix an extra-generalized arithmetical numerical semigroup
$$S = \<a, ah + d, ah + 2d, \ldots, ah + kd\>,$$
and write $a - 1 = qk + r$ for $q, r \in \ZZ_{\ge 0}$ with $r < k$.  
\begin{enumerate}[(a)]
\item 
Each nonzero element $a_i \in \Ap(S;a)$ has the form
$$a_i = x_iah + ((x_i - 1)k + y_i)d$$
for either $x_i \in [1, q]$ and $y_i \in [1,k]$, or $x_i = q + 1$ and $y_i \in [1,r]$.  

\item 
We have $a_i \prec a_j$ in $\Ap(S;a)$ if and only if $x_i < x_j$ and $y_i \ge y_j$.  

\item 
An element $a_j$ covers $a_i$ in $\Ap(S;a)$ if and only if $x_j = x_i + 1$ and $y_i \ge y_j$.  

\end{enumerate}
\end{thm}

\begin{proof}
Part~(a) follows from Proposition~\ref{p:arithmembercrit}(b).  For part~(c), if $x_j = x_i + 1$ and $y_i \ge y_j$, then $a_j - a_i = ah + (k + y_j - y_i)d$ is a minimal generator of $S$, which by Theorem~\ref{t:groupconefacelattice}(c) implies $a_j$ covers $a_i$.  Conversely, suppose $a_j$ covers $a_i$, which by Theorem~\ref{t:groupconefacelattice}(c) means $a_j - a_i = ah + md$ with $1 \le m \le k$.  With $a_i$ and $a_j$ written as in part~(a), we see $x_j - x_i = 1$ and $y_i - y_j = k - m \ge 0$, as desired.  

Lastly, for part~(b), we cannot have $a_i \prec a_j$ unless $x_i < x_j$ by part~(c).  In this case,
$$a_j - a_i = (x_j - x_i)ah + ((x_j - x_i)k + y_j - y_i)d$$
with 
$0 \le (x_j - x_i)k + y_j - y_i \le a - 1$,
so $a_j - a_i \in  \Ap(S;a)$ by Proposition~\ref{p:arithmembercrit}(b) when
$$
\big\lceil \tfrac{1}{k}((x_j - x_i)k + y_j - y_i) \big\rceil
= x_j - x_i + \big\lceil \tfrac{1}{k}(y_j - y_i) \big\rceil
= x_j - x_i,
$$
which happens precisely when $y_i \ge y_j$.  
\end{proof}

\begin{cor}\label{c:arithposet}
Resume notation from Theorem~\ref{t:arithposet}, and suppose $S$ has Kunz poset~$P$.  
\begin{enumerate}[(a)]
\item 
The elements of $P$ have the form $[a_i] = [md]$, where $m = ((x_i - 1)k + y_i)$ takes each integer value in $[0, a - 1]$.  

\item 
The Kunz poset $P$ is graded, with each $[a_i]$ occuring with height $x_i$.  

\item 
The Kunz poset $P$ depends only on $a$, $k$, and the residue of $d$ modulo $a$.  


\end{enumerate}
\end{cor}

\begin{proof}
Part~(a) follows from Theorem~\ref{t:arithposet}(a) and the division algorithm, and part~(b) follows from Theorem~\ref{t:arithposet}(c).  Lastly, part~(c) follows by examining the statements of parts~(a) and~(b).  
\end{proof}

\begin{remark}\label{r:arithposetdrawing}
Again resuming notation from Theorem~\ref{t:arithposet}, the poset $P$ can be drawn so that the nonzero elements form a grid with $k$ columns where $y_j$ specifies the position from the left within each row, and cover relations between sequential rows are drawn precisely when they do not have positive slope.  Additionally, reading elements left to right and bottom to top yields $[0], [d], [2d], \ldots, [(a-1)d]$, and all rows have $k$~elements except possibly the top row (where $x_i = q + 1$), which has $r$ elements.  
\end{remark}

We close this section with a formula for the Frobenius number of an extra-generalized arithmetical numerical semigroup, extending \cite[Theorem~2.8]{omidalirahmati} to the case when $d < 0$.  Our proof, including for positive $d$, utilizes the poset structure from Theorem~\ref{t:arithposet}.  

\begin{cor}\label{c:extraarithfrob}
For a given extra-generalized arithmetical numerical semigroup
$$S = \<a, ah + d, ah + 2d, \ldots, ah + kd\>,$$
we have
$$F(S) = \begin{cases}
\lceil \tfrac{a-1}{k} \rceil ah + (a - 1)d - a & \text{if } d > 0; \\
\lceil \tfrac{a-1}{k} \rceil (ah + kd) + (1-k)d - a & \text{if } d < 0.
\end{cases}$$
\end{cor}

\begin{proof}
Write $a - 1 = qk + r$ for $q, r \in \ZZ_{\ge 0}$ with $r < k$.  
Regardless of whether $d$ is positive or negative, by Theorem~\ref{t:arithposet} and Remark~\ref{r:arithposetdrawing}, $a_j = \max(\Ap(S;a))$ must occur in the top row of the Ap\'ery poset.  If $d > 0$, then $a_j$ is the last element of the top row, meaning either $x_j = q + 1$ and $y_j = r$, or $x_j = q$ and $y_j = k$ (this is the case where $r = 0$), both of which yield
$$F(S) = a_j - a = x_jah + ((x_j - 1)k + y_j)d - a = \big\lceil \tfrac{a-1}{k} \big\rceil ah + (a - 1)d - a.$$
On the other hand, if $d < 0$, then $a_j$ is the first element in the top row, meaning $y = 1$ and either $x_j = q+1$ or $x_j = q$.  In either case, $x_j = \lceil \tfrac{a-1}{k} \rceil$ and 
$$F(S) = a_j - a = x_jah + ((x_j - 1)k + y_j)d - a = \big\lceil \tfrac{a-1}{k} \big\rceil (ah + kd) + (1-k)d - a,$$
as desired.  
\end{proof}

\section{Polyhedra faces containing arithmetical numerical semigroups}
\label{sec:arithfaces}

Having now characterized the Kunz posets of extra-generalized arithmetical numerical semigroups, we examine the geometric properties of faces containing such semigroups.  In particular, we characterize their dimension (Theorem~\ref{t:arithfacedim}) and, for some, their defining rays (Theorem~\ref{t:arithrays}).  Additionally, we prove that in most cases, the faces contain only extra-generalized numerical semigroups (Theorem~\ref{t:onlyarithfaces})

We begin by describing the orbits (in the sense of Remark~\ref{r:automorphisms}) of faces containing extra-generalized arithmetical numerical semigroups, allowing us to restrict some arguments to the case when $d \equiv 1 \bmod a$.  

\begin{lemma}\label{l:arithautomorphisms}
Fix an extra-generalized arithmetical numerical semigroup
$$S = \<a, ah + d, ah + 2d, \ldots, ah + kd\>$$
with containing face $F \subset \mathcal C(\ZZ_a)$.  
Applying any automorphism of $C(\ZZ_a)$ induced by multiplication by some $u \in \ZZ_a^*$ to $F$ yields the face $F'$.  Moreover, $F'$ contains
$$S' = \<a, ah + d', ah + 2d', \ldots, ah + kd'\>$$
for any positive $d'$ with $d' \equiv ud \bmod a$.  
\end{lemma}

\begin{proof}
Multiply each element of the Kunz poset of $S$ by $u$ and apply Corollary~\ref{c:arithposet}.  
\end{proof}

\begin{thm}\label{t:arithfacedim}
Given an extra-generalized arithmetical numerical semigroup
$$S = \<a, ah + d, ah + 2d, \ldots, ah + kd\>$$
with containing face $F \subset C(\ZZ_a)$, we have
$$\dim F = \begin{cases}
a - 1                        & \text{if } k = a - 1; \\
\lfloor \tfrac{a}{2} \rfloor & \text{if } k = a - 2; \\
2                            & \text{if } 1 < k < a - 2; \\
1                            & \text{if } k = 1.
\end{cases}$$
\end{thm}

\begin{proof}
By Lemma~\ref{l:arithautomorphisms}, it suffices to assume $d \equiv 1 \bmod a$.  First, $k = 1$ implies $\mathsf e(S) = 2$, so $S$ lies on the ray $(1, 2, \ldots, a - 1)$, and if $k = a - 1$, then $S$~has maximal embedding dimension, meaning $\dim F = a - 1$.  Next, suppose $k = a - 2$.  By~Corollary~\ref{c:arithposet}, the facet equations of $F$ each have the form $x_i + x_{a-1-i} = x_{a-1}$ for $i = 1, 2, \ldots, \lfloor \tfrac{a - 1}{2} \rfloor$.  Since the matrices
$$
\begin{bmatrix}
1 & 0 & \cdots & 0 & 0 & \cdots & 0 & 1 & -1 \\
0 & 1 & \cdots & 0 & 0 & \cdots & 1 & 0 & -1 \\
\vdots & \vdots & \ddots & \vdots & \vdots & \iddots & \vdots & \vdots & \vdots \\
0 & 0 & \cdots & 1 & 1 & \cdots & 0 & 0 & -1 \\
\end{bmatrix}
\qquad \text{and} \qquad
\begin{bmatrix}
1 & 0 & \cdots & 0 & 0 & 0 & \cdots & 0 & 1 & -1 \\
0 & 1 & \cdots & 0 & 0 & 0 & \cdots & 1 & 0 & -1 \\
\vdots & \vdots & \ddots & \vdots & \vdots & \vdots & \iddots & \vdots & \vdots & \vdots \\
0 & 0 & \cdots & 1 & 0 & 1 & \cdots & 0 & 0 & -1 \\
0 & 0 & \cdots & 0 & 2 & 0 & \cdots & 0 & 0 & -1 \\
\end{bmatrix}
$$
both have full rank $\lfloor \tfrac{a-1}{2} \rfloor$, we conclude $\dim F = a - 1 - \lfloor \tfrac{a-1}{2} \rfloor = \lfloor \tfrac{a}{2} \rfloor$.  

Lastly, suppose $1 < k < a - 2$.  Again applying Corollary~\ref{c:arithposet}, 
we obtain facet equations of the form
$$x_{k+1} = x_1 + x_k = x_2 + x_{k-1} = \cdots \qquad \text{and} \qquad x_{k+2} = x_2 + x_k = x_3 + x_{k-1} = \cdots.$$
Subtracting corresponding equations above yields
$$x_2 - x_1 = x_3 - x_2 = \cdots = x_{k-1} - x_{k-2} = x_k - x_{k-1},$$
meaning that the values of $x_1$ and $x_2$ determine the values for $x_3 \ldots, x_k$, and thus for the remainder of the coordinates as well.  This proves $\dim F \le 2$, and since the coordinates in $C(\ZZ_a)$ of the semigroups $S$, 
$$S' = \<a, a(h+1) + d, \ldots, a(h+1) + kd\>,
\quad \text{and} \quad
S'' = \<a, ah + (d + a), \ldots, ah + k(d + a)\>,$$
are affine independent and lie in $F$ by Corollary~\ref{c:arithposet}(c), we conclude $\dim F = 2$.  
\end{proof}

Our next result implies that, outside of the first 2 cases in Theorem~\ref{t:arithfacedim}, faces containing extra-generalized arithmetical numerical semigroups contain exclusively such semigroups.  

\begin{thm}\label{t:onlyarithfaces}
Fix an extra-generalized arithmetical numerical semigroup
$$S = \<a, ah + d, ah + 2d, \ldots, ah + kd\>.$$
If $k < a - 2$, then any numerical semigroup $S'$ with identical Kunz poset to $S$ is also an extra-generalized arithmetical numerical semigroup.  In particular, the face of $C(\ZZ_a)$ containing $S$ contains only extra-generalized arithmetical numerical semigroups.  
\end{thm}

\begin{proof}
If $S$ and $S'$ have identical Kunz poset $P$, then $m(S') = a$ is the number of elements of $P$ and $e(S') = k+1$ is one more than the number of atoms of $P$ by Theorem~\ref{t:groupconefacelattice}(c), so write $S' = \<a, n_1, \ldots, n_k\>$ with each $n_i \equiv di \bmod a$.  Since $k < a - 2$, $\Ap(S';a)$ has at least 2 more elements $a_{k+1} \equiv (k+1)d \bmod a$ and $a_{k+2} \equiv (k+2)d \bmod a$.  By~Theorem~\ref{t:arithposet}(c), 
$$a_{k+1} = n_1 + n_k = n_2 + n_{k-1} = \cdots \qquad \text{and} \qquad a_{k+2} = n_2 + n_k = n_3 + n_{k-1} = \cdots$$
must all hold.  Let $d' = a_{k+2} - a_{k+1}$.  Subtracting corresponding equations above yields
$$d' = a_{k+2} - a_{k+1} = n_2 - n_1 = n_3 - n_2 = \cdots$$
as well as
$$d' = a_{k+2} - a_{k+1} = n_k - n_{k-1} = n_{k-1} - n_{k-2} = \cdots.$$
If $k = 2j$ is even, then $a_{k+2} = 2n_{j+1} = n_j + n_{j+2}$ and $a_{k+1} = n_j + n_{j+1}$, so $d' = n_j - n_{j-1}$ and $d' = n_{j+1} - n_j$ both appear above.  
If $k = 2j - 1$ is odd, then $a_{k+2} = n_j + n_{j+1}$ and $a_{k+1} = 2n_j = n_{j-1} + n_{j+1}$, so again $d' = n_j - n_{j-1}$ and $d' = n_{j+1} - n_j$ both appear~above.  
Putting everything together, we must have $d \equiv d' \bmod a$, so we can write $n_1 = ah' + d'$ for some $h' \ge 1$, and thus each $n_i = ah' + id'$, as desired.  
\end{proof}

\begin{example}\label{e:arithonly}
The hypothesis $k < a - 2$ is necessary in Theorem~\ref{t:onlyarithfaces}.  If $k = a - 1$, such as for the semigroup $S = \<6, 7, 8, 9, 10, 11\>$, then $S$ is max embedding dimension, so there are ample examples of other numerical semigroups (e.g., $S_1' = \<6,8,10,13,15,17\>$) with identical Kunz poset (depicted in Figure~\ref{f:nonarith1}).  If $k = a - 2$, such as for the numerical semigroup $S_2 = \<6,13,14,15,16\>$, then $S$ shares its Kunz poset (depicted in Figure~\ref{f:nonarith2}) with the semigroup $S_2' = \<6,15,16,19,20\>$.  
\end{example}

\begin{figure}[t!]
\begin{center}
\begin{subfigure}[t]{0.30\textwidth}
\begin{center}
\includegraphics[width=1.2in]{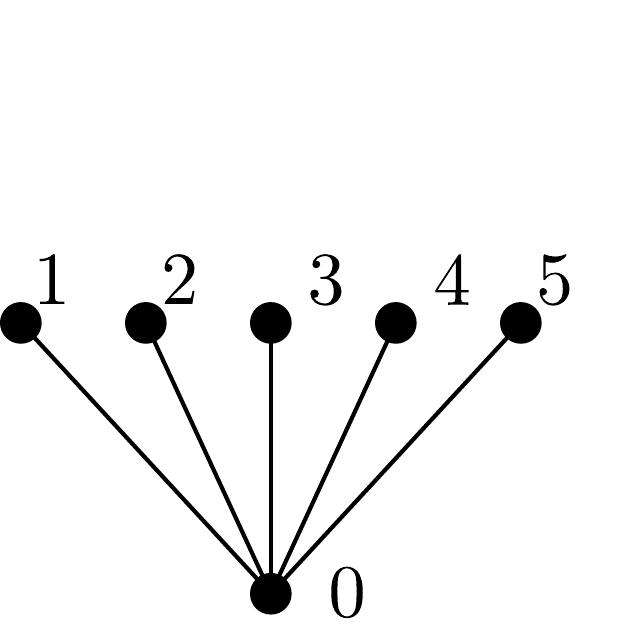}
\end{center}
\caption{}
\label{f:nonarith1}
\end{subfigure}
\hspace{0.02\textwidth}
\begin{subfigure}[t]{0.30\textwidth}
\begin{center}
\includegraphics[width=1.2in]{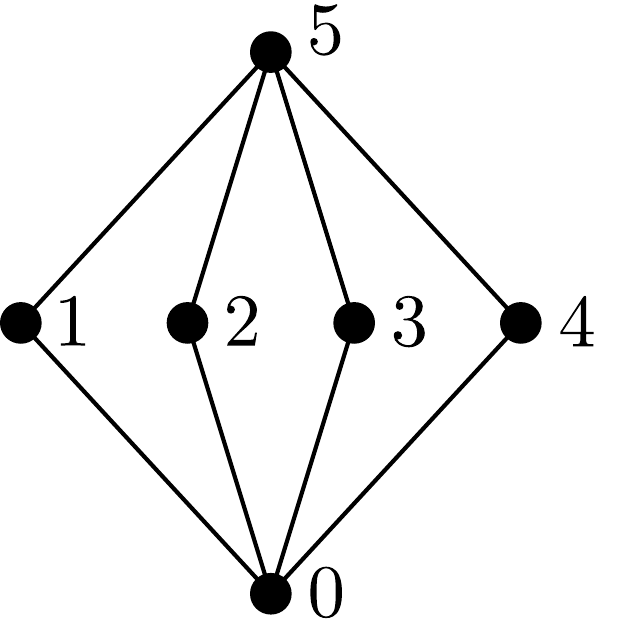}
\end{center}
\caption{}
\label{f:nonarith2}
\end{subfigure}
\hspace{0.02\textwidth}
\begin{subfigure}[t]{0.30\textwidth}
\begin{center}
\includegraphics[width=1.2in]{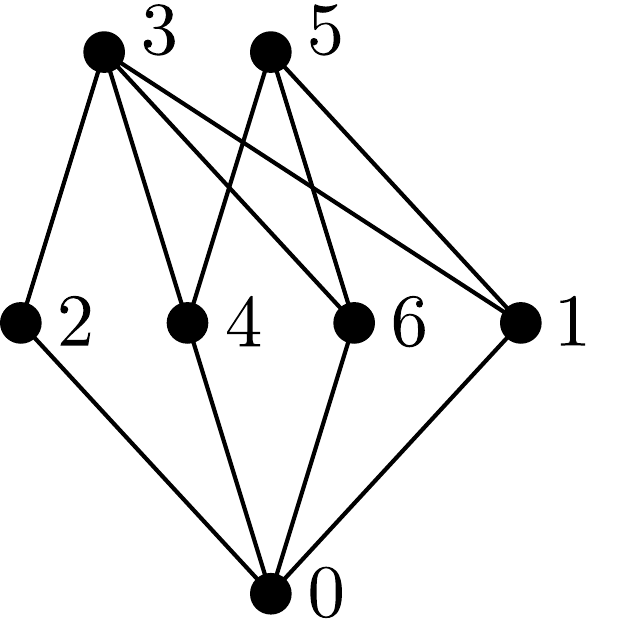}
\end{center}
\caption{}
\label{f:nonarith3}
\end{subfigure}
\end{center}
\caption{Kunz posets for $S_1$ and $S_1'$ (left) and for $S_2$ and $S_2'$ (middle) from Example~\ref{e:arithonly}, along with the Kunz poset of $S_3 = \<7,23,25,27,29\>$ to which Theorem~\ref{t:onlyarithfaces} does apply.}
\end{figure}

\begin{cor}\label{c:arithalsofaces}
Any face with an extra-generalized arithmetical numerical semigroup 
$$S = \<a, ah + d, ah + 2d, \ldots, ah + kd\>$$
in its interior also contains infinitely many arithmetical numerical semigroups.  
\end{cor}

\begin{proof}
This follows from applying Theorem~\ref{t:arithposet} to $S$ and $S' = \<a, a + d', \ldots, a + kd'\>$ where $d'$ is any positive integer with $d' \equiv d \bmod a$.  
\end{proof}

There is still no known classification of the extremal rays of group cones.  The~machinery developed in~\cite{kunzfaces1} yields a method to prove that a given ray is indeed a ray, but proving that a given list of rays is complete, even for a particular face, is a much more difficult task.  However, it is a fact from polyhedral geometry (see~\cite{ziegler}, for instance) that any 2-dimensional face of a pointed cone has exactly 2 extremal rays.  As our final result in this section, we characterize both bounding rays of the 2-dimensional faces of the group cone containing extra-generalized arithmetical numerical semigroups.  

\begin{thm}\label{t:arithrays}
Fix an extra-generalized arithmetical numerical semigroup 
$$S = \<a, ah + d, ah + 2d, \ldots, ah + kd\>,$$
and suppose $1 < k < a - 2$.  The extremal rays of the face $F \subset \mathcal C(\ZZ_a)$ containing $S$ are spanned by the following:
\begin{enumerate}[(i)]
\item 
the vector $\vec r$ that, when added to the Kunz coordinates of $S$ in $P_a$, yields 
$$\<a, ah + (d + a), \ldots, ah + k(d + a)\>;$$
and

\item 
the vector $\vec t$ that, when added to the Kunz coordinates of $S$ in $P_a$, yields
$$\<a, a(h + a) + (d - a\lfloor a/k \rfloor), \ldots, a(h + a) + k(d - a\lfloor a/k \rfloor)\>.$$

\end{enumerate}
The ray $\vec r$ contains all of the numerical semigroups $\<a, b\>$ for positive $b \equiv d \bmod a$, and the ray $\vec t$ contains numerical semigroups if $k \mid (a - 1)$, in which case those semigroups have the form $\<a, b\>$ for some positive $b \equiv kd \bmod a$.  
\end{thm}

\begin{proof}
By Lemma~\ref{l:arithautomorphisms}, it suffices to assume $d = 1$.  Under this assumption, it is easy to check that $\vec r = (1, 2, \dots, a - 1)$, and that $\vec t$ is given by
$$t_i = x_ia - ((x_i - 1)k + y_i) \lfloor a/k \rfloor,$$
with $x_i$ and $y_i$ defined as in Theorem~\ref{t:arithposet}(a).  Note each $t_i$ is surely non-negative, as 
$$((x_i - 1)k + y_i) \lfloor a/k \rfloor \le x_ik \lfloor a/k \rfloor \le x_ia.$$
Since every integer point in $F$ is the Kunz tuple of some extra-generalized arithmetical numerical semigroup by Theorem~\ref{t:onlyarithfaces}, adding $\vec r$ or $\vec t$ to any integer point in $F$ yields another point in $F$, so the rays $\vec r$ and $\vec t$ both lie in~$F$.  

Now, the coordinates of $\vec r$ satisfy the equation $2r_1 = r_2$, which is not satisfied by the Ap\'ery coordinates of $S$ since $k > 1$.  To prove that $\vec t$ is also an extremal ray, we~consider two cases.  First, if $k \mid a$, then 
$$t_k = a - k \lfloor a/k \rfloor = 0,$$
so $\vec t$ lies in a face with nontrivial Kunz subgroup.  Otherwise, we can write $a - 1 = qk + r$ with $q = \lfloor a/k \rfloor$ and $0 \le r \le k - 2$ and obtain
$$t_{(q+1)k}
= t_{k-(r+1)}
= a - (k - (r + 1))q
= a - (k - (a - qk))q
= (a - kq) + (qa - q^2k)
= t_k + t_{qk}$$
as a facet equation not satisfied by the Ap\'ery coordinates of $S$.  
As such, we conclude $\vec r$ and $\vec t$ both lie in proper faces of $F$, which necessarily are rays since $\dim F = 2$.  

For the final claims, it is clear that $\vec r = (1, 2, \ldots, a - 1)$ contains semigroups of the form $\<a, b\>$ for $b \equiv 1 \bmod a$ since the Kunz poset is a total ordering.  Likewise, if $a - 1 = qk$ for $q \in \ZZ_{\ge 0}$, then $t_k = 1$ and 
$$t_{k-i} = a - (k - i)q = 1 + iq$$
for each $i = 1, \ldots, k - 1$, so the Kunz poset of $\vec t$ is also a total ordering with unique atom $k$.  This completes the proof.  
\end{proof}

\begin{example}\label{e:arithrays}
The Kunz poset of $S = \<13,14,15,16,17\>$ is depicted in Figure~\ref{f:arithraysboth}, along with the posets of its bounding rays in Figures~\ref{f:arithray1} and~\ref{f:arithray2}.  By Theorem~\ref{t:arithrays}, the Kunz poset of the first ray will always be a total order, obtained by ``reading across the rows'' of the Kunz poset for $S$.  For this particular semigroup $k \mid (a - 1)$, so the latter ray is also a total order, obtained by reading ``bottom to top and right to left'' in the Kunz poset of $S$.  The condition $k \mid (a - 1)$ holds whenever the top row of the Kunz poset has the full $k$ elements.  
\end{example}

\begin{example}\label{e:arithbadray}
When $k \nmid (a-1)$, the Kunz poset of the ray $\vec t$ in Theorem~\ref{t:arithrays} can be substantially more complicated.  One such example, for $S = \<16,23,30,37,44,51,58\>$, is depicted in Figure~\ref{f:arithraybad}.  Generally speaking, the ray $\vec t$ often does not contain numerical semigroups, and sometimes corresponds to a nontrivial subgroup (for instance, it is not hard to show using Theorem~\ref{t:arithrays} that this happens if $k \mid a$).  
\end{example}

\begin{figure}[t!]
\begin{center}
\begin{subfigure}[t]{0.15\textwidth}
\begin{center}
\includegraphics[height=2.5in]{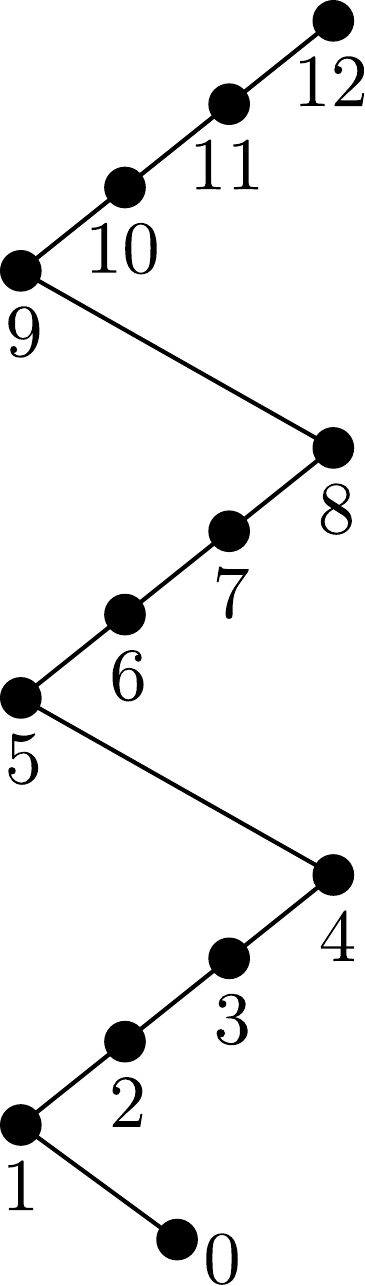}
\end{center}
\caption{}
\label{f:arithray1}
\end{subfigure}
\hspace{0.02\textwidth}
\begin{subfigure}[t]{0.20\textwidth}
\begin{center}
\includegraphics[height=2.5in]{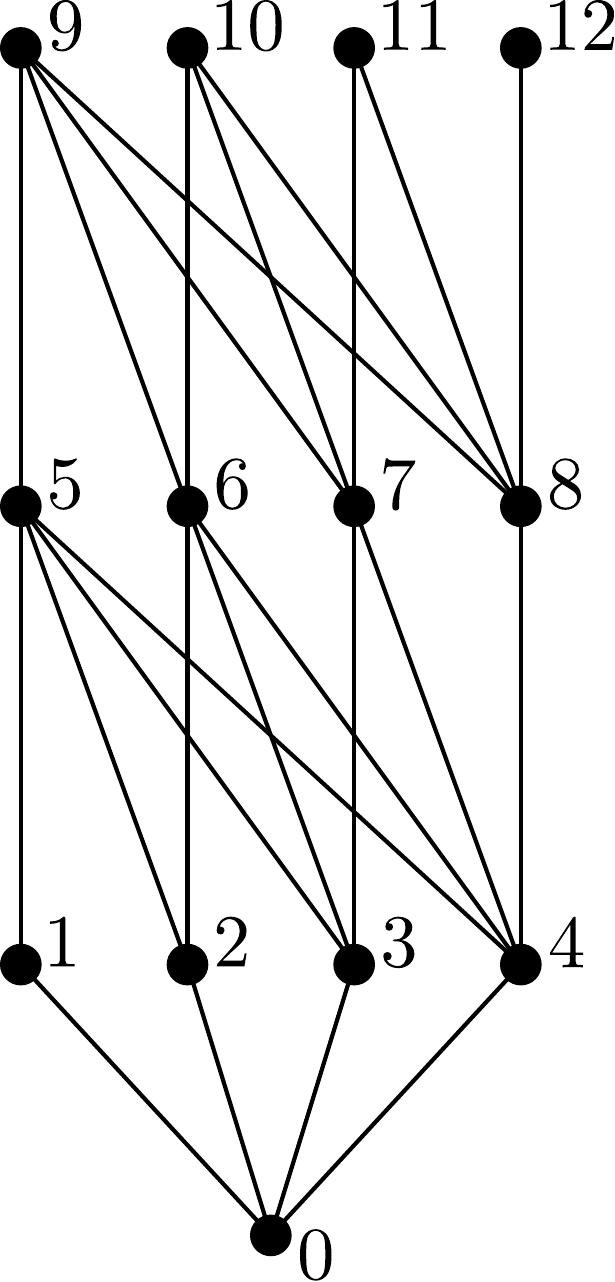}
\end{center}
\caption{}
\label{f:arithraysboth}
\end{subfigure}
\hspace{0.02\textwidth}
\begin{subfigure}[t]{0.15\textwidth}
\begin{center}
\includegraphics[height=2.5in]{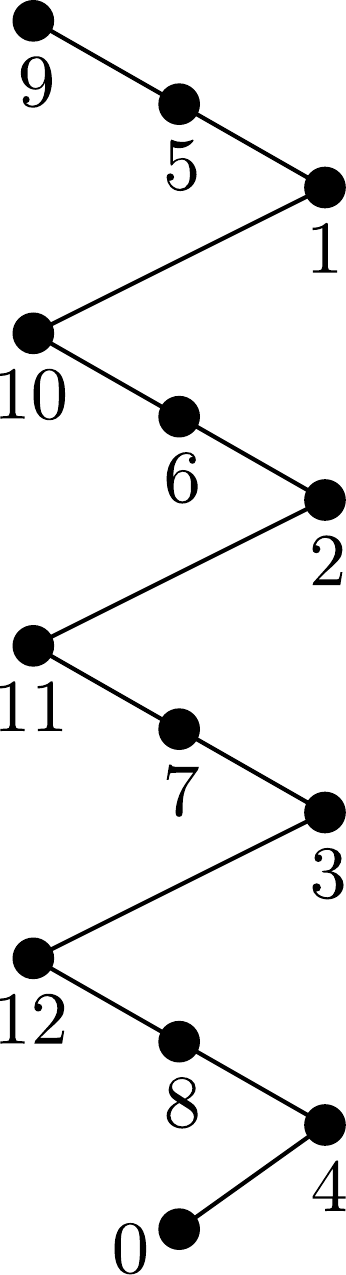}
\end{center}
\caption{}
\label{f:arithray2}
\end{subfigure}
\hspace{0.10\textwidth}
\begin{subfigure}[t]{0.30\textwidth}
\begin{center}
\includegraphics[height=2.5in]{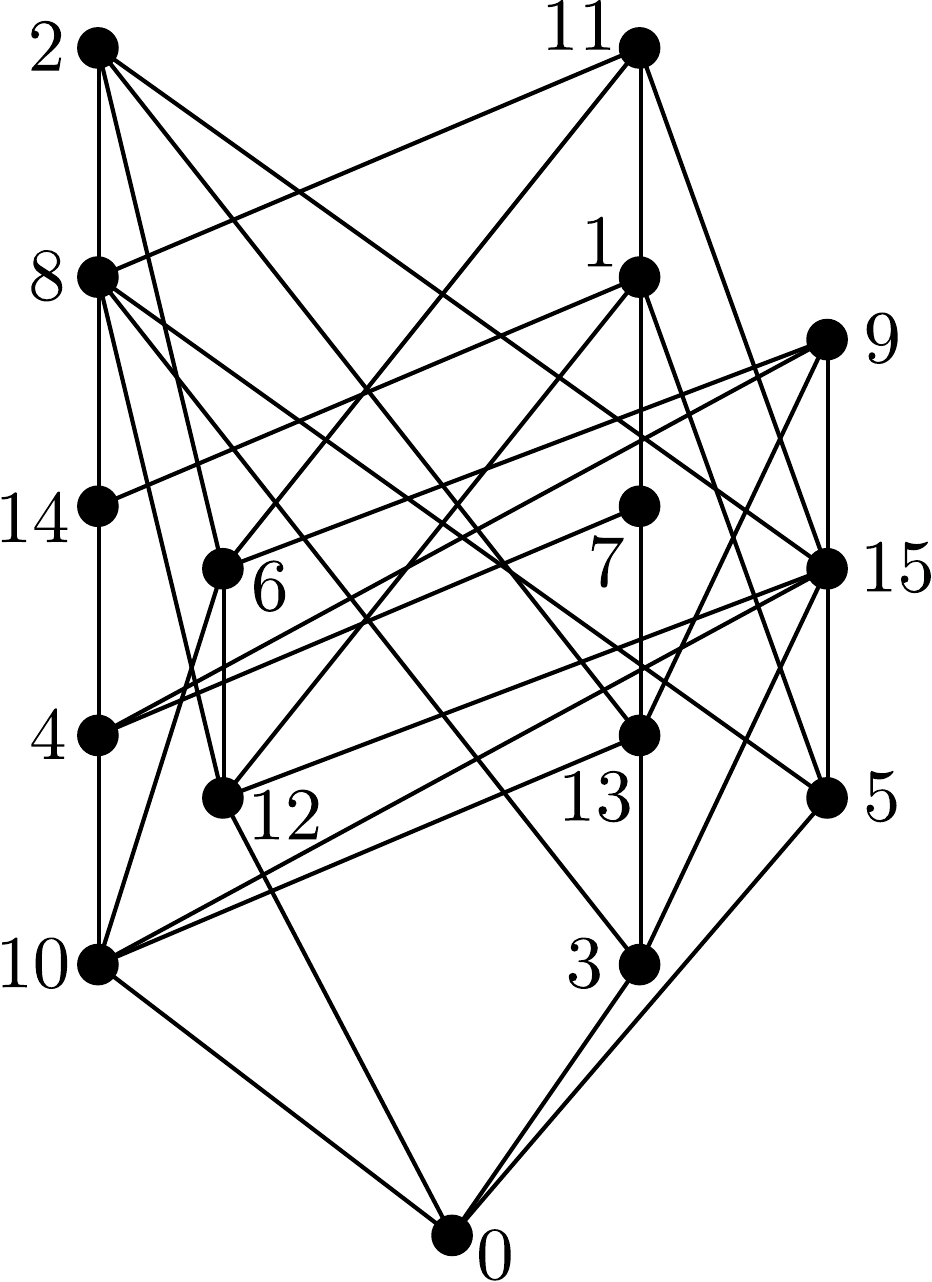}
\end{center}
\caption{}
\label{f:arithraybad}
\end{subfigure}
\end{center}
\caption{The Kunz poset of $S = \<13,14,15,16,17\>$ ((b) above) and its two bounding rays ((a) and (c) above), as in Example~\ref{e:arithrays}.  The final Kunz poset is that of a bounding ray for $S = \<16,23,30,37,44,51,58\>$ discussed in Example~\ref{e:arithbadray}.}
\end{figure}

\begin{remark}\label{r:remainingarithrays}
The remaining faces of $C(\ZZ_a)$ containing extra-generalized arithmetical numerical semigroups have substantially more extremal rays than those described in Theorem~\ref{t:arithrays}.  Indeed, if $a = 19$ and $k = 17$, then each such face has $726$ rays, while $P_{19}$ itself has a grand total of $11$,$665$,$781$ rays~\cite{wilfmultiplicity}.  
\end{remark}

\section{Posets of monoscopic numerical semigroups}
\label{sec:monoscopic}

In this section, we introduce monoscopic numerical semigroups (Definition~\ref{d:monoscopic}), and extend a known characterization of the Ap\'ery set of monoscopic numerical semigroups to a characterization of their Ap\'ery posets (Theorem~\ref{t:monoscopicposet}) and Kunz posets (Corollary~\ref{c:monoscopicposet}).  

\begin{defn}\label{d:monoscopic}
Fix a numerical semigroup $S = \<n_1, \ldots, n_k\>$, an integer $\beta \in \ZZ_{\ge 2}$, and an element $\alpha \in S \setminus \{n_1, \ldots, n_k\}$ with $\gcd(\alpha, \beta) = 1$.  The semigroup 
$$T = \<\alpha\> + \beta S = \<\alpha, \beta n_1, \ldots, \beta n_k\>$$
is called a \emph{monoscopic gluing} of $S$, or simply \emph{monoscopic}.  
\end{defn}

It is well known~\cite[Lemma~8.8]{numerical} that under the given conditions, the generating set for $T$ given in Definition~\ref{d:monoscopic} is minimal.  

\begin{thm}[{\cite[Theorem~8.2]{numerical}}]\label{t:monoscopicaperyset}
Suppose $S = \<m, n_2, \ldots, n_k\>$ and that $T = \<\alpha\> + \beta S$ is a monoscopic gluing.  We have 
$$\Ap(T;\beta m) = \{b\alpha + a\beta : a \in \Ap(S;m) \text{ and } 0 \le b \le \beta - 1\}.$$
\end{thm}

\begin{example}\label{e:monoscopicextensions}
Let $S = \<4, 13, 18\>$, and consider the monoscopic gluings 
$$T_1 = \<43\> + 3S = \<12, 39, 43, 54\> \qquad \text{and} \qquad T_2 = \<31\> + 3S = \<12, 31, 39, 54\>$$
of $S$.  Although the Ap\'ery sets of $T_1$ and $T_2$ have identical structure by Theorem~\ref{t:monoscopicaperyset}, and have nearly identical Kunz posets, as depicted in Figure~\ref{f:monoscopicextensions}, there is a subtle distinction.  The key turns out to be that $31 \in \Ap(S;4)$, meaning $\alpha = 31$ is the smallest possible value $\alpha$ can take in its equivalence class modulo $12$.  When we examine where monoscopic semigroups lie in the Kunz polyhedron in Section~\ref{sec:monoscopicembeddings}, we will see that varying $\alpha$ within its equivalence class modulo $\beta m$ yields semigroups within the same face, unless $\alpha \in \Ap(S;m)$, which places the resulting semigroup on a boundary face.  
\end{example}

\begin{figure}[t!]
\begin{center}
\includegraphics[width=2.0in]{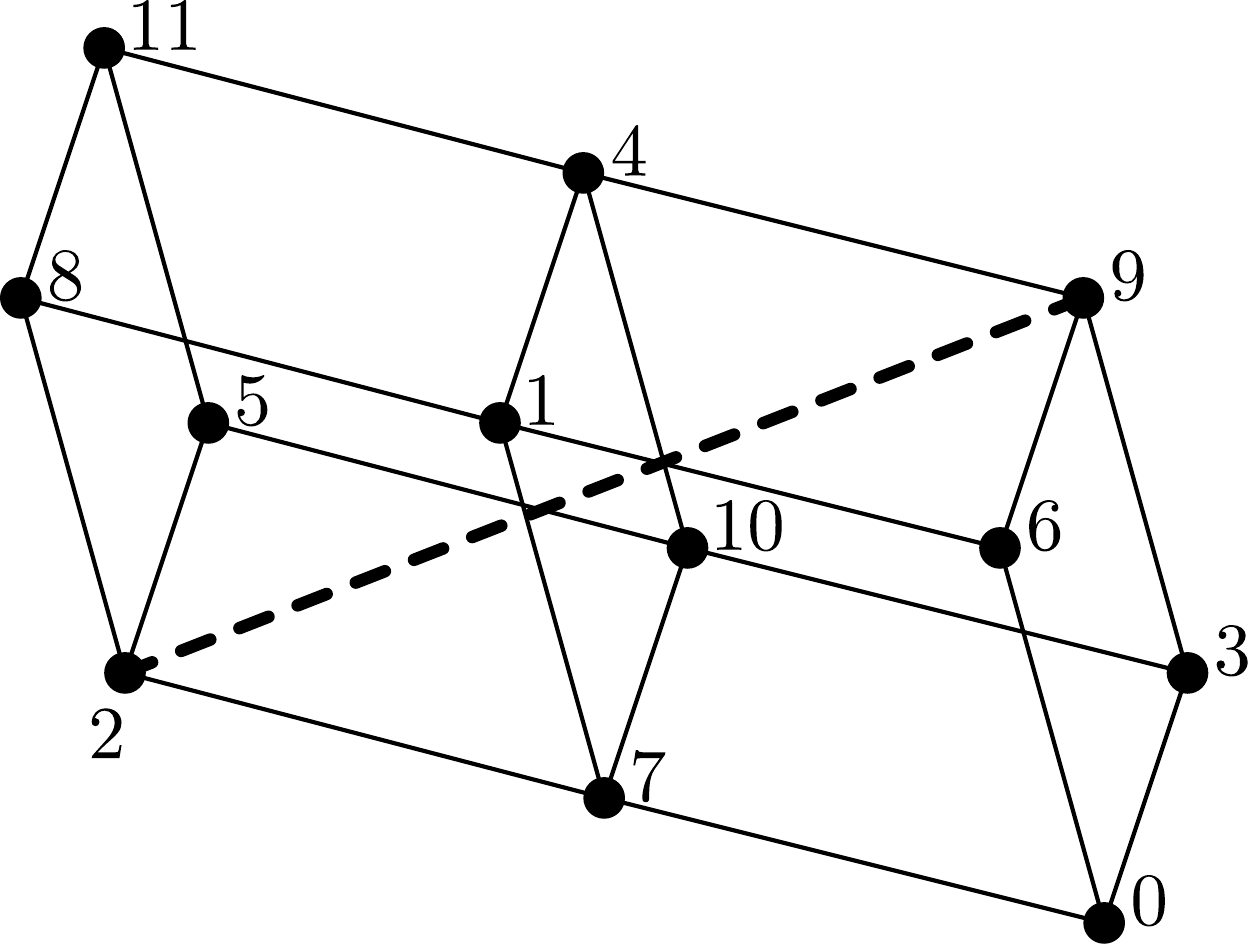}
\end{center}
\caption{The Kunz poset for $T_1 = \<12, 39, 43, 54\>$ (without the dashed edge) and $T_2 = \<12, 31, 39, 54\>$ (with the dashed edge) from Example~\ref{e:monoscopicextensions}, both of which are monoscopic gluings of $S = \<4, 13, 18\>$.}
\label{f:monoscopicextensions}
\end{figure}

\begin{thm}\label{t:monoscopicposet}
Let $S = \<m, n_2, \ldots, n_k\>$.  Suppose $T = \<\alpha\> + \beta S$ is a monoscopic gluing, and fix $b\alpha + a\beta$, $b'\alpha + a'\beta \in \Ap(T;\beta m)$.  
\begin{enumerate}[(a)]
\item 
If $\alpha \notin \Ap(S;m)$, then $b\alpha + a\beta \preceq_T b'\alpha + a'\beta$ if and only if $a \preceq_S a'$ and $b \le b'$.  

\item 
If $\alpha \in \Ap(S;m)$, then $b\alpha + a\beta \preceq_T b'\alpha + a'\beta$ if and only if $a \preceq_S a'$ and either $b \le b'$ or $\alpha \preceq_S a' - a$.  

\end{enumerate}
\end{thm}

\begin{proof}
It is clear, in either case, that if $a \preceq_S a'$ and $b \le b'$, then $b\alpha + a\beta \preceq_T b'\alpha + a'\beta$.  On the other hand, if $b > b'$, but $\alpha \in \Ap(S;m)$ and $\beta\alpha \preceq_S a' - a$, then 
$$(b'\alpha + a'\beta) - (b\alpha + a\beta) = (b' - b)\alpha + (a' - a)\beta = (b' - b + \beta)\alpha + (a' - a - \alpha)\beta \in T,$$
so again $b\alpha + a\beta \preceq_T b'\alpha + a'\beta$.  

Conversely, suppose $b\alpha + a\beta \preceq_T b'\alpha + a'\beta$, meaning $(b' - b)\alpha + (a' - a)\beta \in \Ap(T;\beta m)$.  By Theorem~\ref{t:monoscopicaperyset}, this means 
$$(b' - b)\alpha + (a' - a)\beta = b''\alpha + a''\beta$$
with $a'' \in \Ap(S;m)$ and $0 \le b'' \le \beta - 1$.  Rearranging this equality yields 
$$(b' - b - b'')\alpha + (a' - a - a'')\beta = 0,$$
and since $\gcd(\alpha, \beta) = 1$, we must have $\beta \mid (b' - b - b'')$ and $\alpha \mid (a' - a - a'')$.  Given the bounds on $b$, $b'$, and $b''$, this is only possible if $b + b'' - b' = 0$ or $b + b'' - b' = \beta$.  If the former holds, then $b' - b = b'' \ge 0$ and $a' - a = a'' \in S$, so we are done.  If the latter holds, then $a' - a - \alpha = a'' \in S$, which is only possible if $\alpha \in \Ap(S;m)$ as well.  
\end{proof}

\begin{cor}\label{c:monoscopicposet}
If $S = \<m, n_2, \ldots, n_k\>$ and $T = \<\alpha\> + \beta S$ is a monoscopic gluing, then $b'\alpha + a'\beta$ covers $b\alpha + a\beta$ in $\Ap(T;\beta m)$ if and only if one of the following holds:
\begin{enumerate}[(i)]
\item 
$a = a'$ and $b' - b = 1$;

\item 
$b = b'$ and $a' - a = n_j$ for some $j \ge 2$; or

\item 
$b = \beta - 1$, $b' = 0$, and $a' - a = \alpha$.  

\end{enumerate}
Moreover, condition~(iii) occurs in $\Ap(T;\beta m)$ if and only if $\alpha \in \Ap(S;m)$.  
\end{cor}

\begin{proof}
By Theorem~\ref{t:groupconefacelattice}(c), $b'\alpha + a'\beta$ covers $b\alpha + a\beta$ in $\Ap(T;\beta m)$ if and only if their difference equals either $\alpha$ or $\beta n_j$ for some $j \ge 2$.  By Theorem~\ref{t:monoscopicposet}, the latter case forces~(ii) to hold, and the former case holds precisely when either~(i) or~(iii) holds.  
\end{proof}

\section{Monoscopic embeddings of polyhedra}
\label{sec:monoscopicembeddings}

In this section, we characterize the faces of the Kunz polyhedron $P_m$ containing monoscopic numerical semigroups, as well as all lower dimensional faces contained therein.  We do so by constructing a family of combinatorial embeddings (Definition~\ref{d:monoscopicembedding}) of group cones, which we show in Theorem~\ref{t:augmonoscopicfaces} induces an injection of face lattices.  The faces of $P_m$ corresponding to these faces, together with those described in Theorem~\ref{t:monoscopicfaces}, contain all of the monoscopic numerical semigroups, and only monoscopic numerical semigroups (Theorem~\ref{t:onlymonoscopicfaces}).  

Figure~\ref{f:monoscopicfacelattice} depicts the portion of the face lattice of $\mathcal C(\ZZ_{12})$ described by Theorems~\ref{t:augmonoscopicfaces} and~\ref{t:monoscopicfaces}, including all posets therein.  

\begin{notation}\label{n:phantomzero}
In order to simplify numerous expressions in the remainder of the paper, we adopt the convention of prepending a ``$0$'' entry to each point in $\mathcal C(G)$, indexed by the identity element of $G$.  More precisely, we write each $(x_1, x_2, \ldots) \in \mathcal C(G)$ in the form $(x_0, x_1, x_2, \ldots)$ with $x_0 = 0$, effectively replacing $\mathcal C(G)$ with $\{0\} \times \mathcal C(G)$.  
\end{notation}

\begin{defn}\label{d:monoscopicembedding}
Fix an Abelian group $G$, a subgroup $H \subset G$ so that $G/H$ is cyclic, and an element $\rho \in G$ whose image in $G/H$ is a generator.  Letting $\beta = |G/H|$, define 
$$
\begin{array}{r@{}c@{}l}
\Phi_\rho:\mathcal C(H) &{}\longrightarrow{}& \mathcal C(G) \\
w &{}\longmapsto{}& x
\end{array}$$
where $x_{a+b\rho} = \beta w_a + b w_{\beta\rho}$ for each $a \in H$ and $0 \le b < \beta$.  We call $\Phi_\rho$ a \emph{monoscopic embedding} of $\mathcal C(H)$ into $\mathcal C(G)$ along $\rho$.  
\end{defn}

\begin{sidewaysfigure}
\vspace{5in}
\includegraphics[width=7.9in]{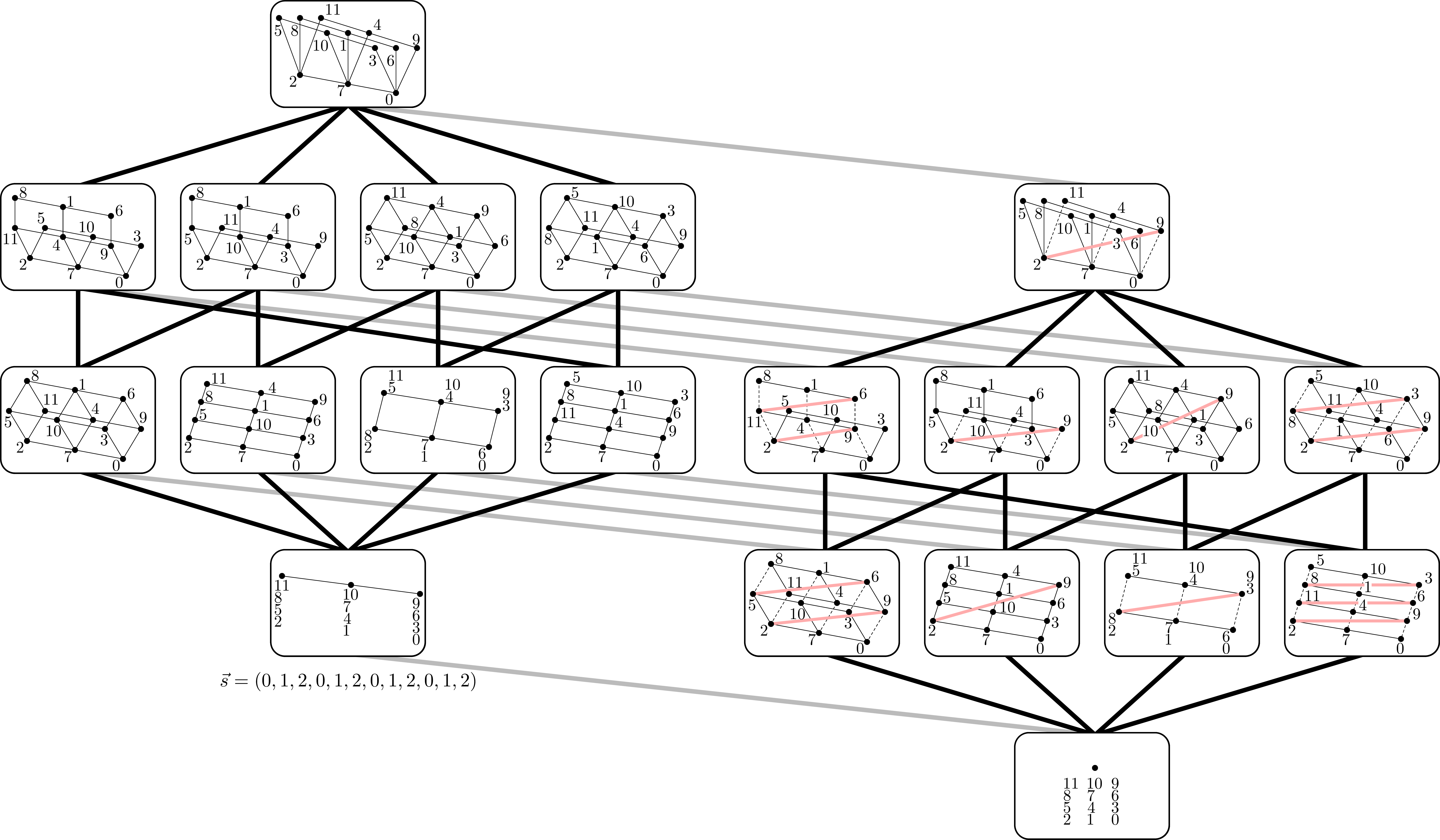}
\caption{The portion of the face lattice of $\mathcal C(\ZZ_{12})$ containing monoscopic numerical semigroups with $H = \<3\>$ and $\rho = 7$.  The right half of the figure consists of the faces in the image of $\mathcal C(\ZZ_4)$ under $\Phi_\rho$, and the posets therein are augmented monoscopic extensions.  The additional cover relations that occur in augmented extensions but not their non-augmented counterparts are marked by thick red lines.  These additional cover relations render some existing cover relations redenduant; these are marked by dashed lines.}
\label{f:monoscopicfacelattice}
\end{sidewaysfigure}

\begin{lemma}\label{l:phimap}
The monoscopic embedding $\Phi_\rho$ is well-defined and injective.  
\end{lemma}

\begin{proof}
Under the given assumptions, every element of $G$ can be written uniquely in the form $a + k\rho$ for some $a \in H$ and $b \in \ZZ$ with $0 \le b \le \beta - 1$.  Moreover, if $w \in \mathcal C(H)$ and $x = \Phi_\rho(w)$, it is easy to check that $x_0 = 0$, and for any $a, a' \in H$ and $0 \le b, b' < \beta$, 
\begin{align*}
x_{a+b\rho} + x_{(a'-a)+(b'-b)\rho}
&= (\beta w_a + b w_{\beta\rho}) + (\beta w_{a'-a} + (b' - b) w_{\beta\rho})
= \beta(w_a + w_{a'-a}) + b' w_{\beta\rho} \\
&\ge \beta w_{a'} + b' w_{\beta\rho}
= x_{a'+b'\rho}
\end{align*}
so the image of $\Phi_\rho$ lies in $\mathcal C(G)$.  This proves $\Phi_\rho$ is well-defined.  From here, it is clear that $\Phi_\rho$ is linear, and projecting the image of $\Phi_\rho$ onto the coordinates indexed by $H$ yields a linear map that simply scales each vector by $\beta$, so $\Phi_\rho$ is injective as well.  
\end{proof}

\begin{defn}\label{d:monoscopicextension}
Fix a poset $P = (H/H', \preceq_P)$, where $H' \subset H$ is some subgroup, and suppose $a, b \in H/H'$ and $0 \le b, b' \le \beta - 1$.  
\begin{enumerate}[(a)]
\item 
The \emph{monoscopic extension} of $P$ along $\rho$ is the poset $Q = (G/H', \preceq_Q)$ satisfying $a + b \rho \preceq_Q a' + b' \rho$ whenever $a \preceq_P a'$ and $b \le b'$.  

\item 
The \emph{augmented monoscopic extension} of $P$ along $\rho$ is the poset $Q$ defined as follows.  
\begin{enumerate}[(i)]
\item 
If $\beta\rho \ne 0$, then $Q = (G/H', \preceq_Q)$ with $a + b \rho \preceq_Q a' + b' \rho$ whenever $a \preceq_P a'$ and either $b \le b'$ or $\beta\rho \preceq_P a' - a$.

\item 
If~$\beta\rho = 0$, then $Q$ is the poset on $G/(H' + \<\rho\>)$ that is identical to $P$ under the natural group isomorphism $G/(H' + \<\rho\>) \iso H/H'$.  

\end{enumerate}
\end{enumerate}
\end{defn}

\begin{remark}\label{r:monoscopicextension}
The (non-augmented) monoscopic extension of a poset $P$ is isomorphic to the Cartesian product of $P$ with a total ordering.  Also, any augmented monoscopic extension is a refinement of the corresponding non-augmented monoscopic extension.  
\end{remark}

\begin{cor}\label{c:monoscopicextensiongluing}
If $S = \<m, n_2, \ldots, n_k\>$, and $T = \<\alpha\> + \beta S$ is a monoscopic gluing, then the Kunz poset of $T$ is the monoscopic extension of the Kunz poset of $S$ along $\rho = \ol \alpha$, one that is augmented if and only if $\alpha \in \Ap(S)$.  
\end{cor}

\begin{proof}
Note that if $\beta\rho = 0$, then $\alpha \in \Ap(S)$ is impossible.  As such, the claim follows from Theorem~\ref{t:monoscopicposet} upon unraveling definitions.  
\end{proof}

\begin{thm}\label{t:augmonoscopicfaces}
The image of $\Phi_\rho$ is a face of $\mathcal C(G)$.  More precisely, given any face $F \subset \mathcal C(H)$ with Kunz poset $P = (H/H', \preceq)$, the image $\Phi_\rho(F)$ is a face of $\mathcal C(G)$ whose Kunz poset is the augmented monoscopic extension $Q$ of $P$ along $\rho$.  
\end{thm}

\begin{proof}
First, fix $w \in F \subset \mathcal C(H)$, let $x = \Phi_\rho(w)$, let $F'$ denote the face containing~$x$, and let $Q$ denote the corresponding Kunz poset of $F'$.  If $\beta\rho = 0$, then $x_{a + b\rho} = 0$ whenever $w_a = 0$ and $bw_{\beta\rho} = 0$.  If $\beta\rho \ne 0$, then this occurs when $a \in H'$ and $b = 0$.  
In either case, $Q$ has the claimed ground set.  

Now, suppose $a, b \in H$ and $0 \le b, b' \le \beta - 1$.  If $\beta\rho = 0$, then 
$$x_{a + b\rho} + x_{(a' - a) + (b' - b)\rho} = \beta w_a + \beta w_{a' - a} \ge \beta w_{a'} = x_{a' + b'\rho}$$
with equality precisely when $a \preceq_P a'$, so assume $\beta\rho \ne 0$.  If $b \le b'$, then
\begin{align*}
x_{a + b\rho} + x_{(a' - a) + (b' - b)\rho}
&= (\beta w_a + bw_{\beta\rho}) + (\beta w_{a' - a} + (b' - b)w_{\beta\rho})
= \beta(w_a + w_{a' - a}) + b'w_{\beta\rho}
\\
&\ge \beta w_{a'} + b'w_{\beta\rho}
= x_{a' + b'\rho},
\end{align*}
with equality precisely when $a \preceq_P a'$.  Alternatively, if $b > b'$, then
$$(a' + b'\rho) - (a + b\rho) = (a' - a - \beta\rho) + (b' - b + \beta)\rho$$
with $0 \le b' - b + \beta \le \beta - 1$, so we have
\begin{align*}
x_{a + b\rho} + x_{(a' - a - \beta\rho) + (b' - b + \beta)\rho}
&= (\beta w_a + bw_{\beta\rho}) + (\beta w_{a' - a - \beta\rho} + (b' - b + \beta)w_{\beta\rho})
\\
&= \beta(w_a + w_{a' - a - \beta\rho} + w_{\beta\rho}) + b'w_{\beta\rho}
\\
&\ge \beta(w_a + w_{a' - a}) + b'w_{\beta\rho}
\\
&\ge \beta w_{a'} + b'w_{\beta\rho}
= x_{a' + b'\rho},
\end{align*}
with equality precisely when $\beta\rho \preceq_P a' - a$ and $a \preceq_P a'$.  In either case, we obtain $a + b\rho \preceq_Q a' + b'\rho$ in the exact cases required by Definition~\ref{d:monoscopicextension}, thereby proving $\Phi_\rho(w)$ lies in the interior of the claimed face.  

Conversely, let $F' \subset \mathcal C(G)$ denote the face whose Kunz poset is the augmented monoscopic extension $Q$ of $P$ (which must exist by the above argument), and fix $x \in F'$.  Defining $w_a = \tfrac{1}{\beta}x_a$ for $a \in H$, we see
\begin{align*}
x_{a + b\rho}
= x_a + bx_\rho
= x_a + \tfrac{1}{\beta}b(x_\rho + x_{(\beta-1)\rho})
= x_a + \tfrac{1}{\beta}bx_{\beta\rho}
= \beta w_a + bw_{\beta\rho}
= \Phi_\rho(w),
\end{align*}
where the first and second equalities hold since $Q$ is an augmented monoscopic extension.  This proves set equality $\Phi_\rho(F) = F'$, thereby completing the proof.  
\end{proof}

\begin{defn}\label{d:betaray}
The \emph{beta ray} $\vec s$ of a monoscopic embedding $\Phi_\rho$ is defined
$$s_{a+b\rho} = b$$
for each $a \in H$ and $0 \le b \le \beta - 1$.  Notice that $s_a = 0$ precisely when $a \in H$, so $\vec s$ must lie in a face whose Kunz subgroup under is $H$.  
\end{defn}

\begin{lemma}\label{l:betarayind}
The beta ray $\vec s$
\begin{enumerate}[(a)]
\item 
lies in a face of $\mathcal C(G)$ whose corresponding subgroup is $H$, and

\item 
is linearly independent to each vector in the image of~$\Phi_\rho$.  

\end{enumerate}
\end{lemma}

\begin{proof}
The first claim is easy to verify.  For the second claim, since $\Phi_\rho$ is linear and $\mathcal C(H)$ is full-dimensional, it suffices to prove $\vec s$ lies outside of the image of $\Phi_\rho$.  Indeed, projecting the image of $\Phi_\rho$ onto the coordinates indexed by $H$ is injective by the proof of Lemma~\ref{l:phimap}, while applying the same projection to $\vec s$ yields $0$.  
\end{proof}

\begin{thm}\label{t:monoscopicfaces}
For any face $F \subset \mathcal C(H)$, the set $\RR_{\ge 0}\vec s + \Phi_\rho(F)$ is a face of $\mathcal C(G)$ whose Kunz poset is the monoscopic extension of the Kunz poset of $F$.  
\end{thm}

\begin{proof}
Let $P = (H/H', \preceq_P)$ denote the Kunz poset of $F$, let $F'$ denote the smallest face containing $\RR_{\ge 0}\vec s + \Phi_\rho(F)$, and let $Q = (G/H'', \preceq_Q)$ denote the Kunz poset of~$F'$.  Since $\Phi_\rho(F) \subset F'$, Theorem~\ref{t:augmonoscopicfaces} implies $H'' \subset H'$, and the coordinates in which $s$ is zero are precisely those indexed by $H$, so we must have $H'' = H'$.  Next, fix $x \in \RR_{\ge 0}\vec s + \Phi_\rho(F)$, and write $x$ in the form $x = y + cs$ for $y \in \Phi_\rho(F)$ and $c \ge 0$.  If~$a + b\rho, a' + b'\rho \in G$ satisfy $a \preceq_P a'$ and $b \le b'$, then by Theorem~\ref{t:augmonoscopicfaces},
\begin{align*}
x_{a + b\rho} + x_{(a' - a) + (b' - b)\rho}
= y_{a + b\rho} + cb + y_{(a' - a) + (b' - b)\rho} + c(b' - b)
= y_{a' + b'\rho} + cb'
= x_{a' + b'\rho}.
\end{align*}
Additionally, among the facet equations satisfied by $\Phi_\rho(F)$, these are the only ones satisfied by $s$.  As such, we conclude $Q$ equals the monoscopic extension of $P$.  

Lastly, fix $x \in F'$.  Let $c = x_\rho - \tfrac{1}{\beta}x_{\beta\rho}$, and write $y = x - cs$.  To complete the proof, we must show $y \in \Phi_\rho(F)$.  Fix $a + b\rho, a' + b'\rho \in G$.  If $b \le b'$, then
\begin{align*}
y_{a + b\rho} + y_{(a' - a) + (b' - b)\rho}
= x_{a + b\rho} - cb + x_{(a' - a) + (b' - b)\rho} - c(b' - b)
\ge x_{a' + b'\rho} - cb'
= y_{a' + b'\rho}
\end{align*}
with equality precisely when $a \preceq_P a'$.  On the other hand, if $b > b'$, then
\begin{align*}
y_{a + b\rho} + y_{(a' - a - \beta\rho) + (b' - b + \beta)\rho} 
&= x_{a + b\rho} - cb + x_{(a' - a - \beta\rho) + (b' - b + \beta)\rho} - c(b' - b + \beta)
\\
&= x_{a + b\rho} + x_{(a' - a - \beta\rho) + (b' - b + \beta)\rho} - c\beta - cb'
\\
&= x_{a + b\rho} + x_{(a' - a - \beta\rho) + (b' - b + \beta)\rho} - \beta x_\rho + x_{\beta\rho} - cb'
\\
&\ge x_{a + b\rho} + x_{(a' - a) + (b' - b + \beta)\rho} - \beta x_\rho - cb'
\\
&= x_{a + b\rho} + x_{(a' - a)} - (b - b') x_\rho - cb'
\\
&\ge x_{a' + b\rho} - (b - b') x_\rho - cb'
\\
&= x_{a' + b'\rho} - cb'
= y_{a' + b'\rho}
\end{align*}
with equality whenever $\beta\rho \preceq_P a' - a$ and $a \preceq_P a'$.  We conclude $y \in \Phi_\rho(F)$.  
\end{proof}

Our final result of this section is a converse of sorts to Corollary~\ref{c:monoscopicextensiongluing}, namely that the faces of the group cone containing monoscopic numerical semigroups contain only monoscopic numerical semigroups.  

\begin{thm}\label{t:onlymonoscopicfaces}
Let $S = \<m, n_2, \ldots, n_k\>$, suppose $T = \<\alpha\> + \beta S$ is a monoscopic gluing, and let $F \subset \mathcal C(\ZZ_{\beta m})$ denote the face containing $T$.  Any numerical semigroup $T'$ in $F$ can be expressed as a monoscopic gluing $T' = \<\alpha'\> + \beta S'$, where $\alpha' \equiv \alpha \bmod \beta m$ and $S'$ is a numerical semigroup on the same face of $\mathcal C(\ZZ_m)$ as $S$.  
\end{thm}

\begin{proof}
Let $\alpha' \in T'$ denote the minimal generator of $T'$ satisfying $\alpha' \equiv \alpha \bmod \beta m$, and let $\rho \in \ZZ_{\beta m}$ denote the equivalence class containing $\alpha$ and $\alpha'$.  The remaining generators of $T'$ must each be divisible by $\beta$, so we can write $T' = \<\alpha', \beta m, \beta n_2', \ldots, \beta n_k'\>$.  
Letting~$S' = \<m, n_2', \ldots, n_k'\>$, we claim $T' = \<\alpha'\> + \beta S'$ is a gluing.  
Indeed, it is clear $\gcd(\alpha, \beta) = 1$ since the above generating set for $T'$ is minimal, so we must show $\alpha' \in S' \setminus \{m, n_2', \ldots, n_k'\}$.  However, this follows from Corollary~\ref{c:monoscopicextensiongluing} since $\beta\alpha'$ can be factored using $\beta n_2', \ldots, \beta n_k'$.  
This proves the claim.  

It remains to show that $S'$ lies in the same face of $\mathcal C(\ZZ_m)$ as $S$.  Again, Corollary~\ref{c:monoscopicextensiongluing} implies every element of $\Ap(T';\beta m)$ divisible by $\beta$ can be factored using $\beta n_2', \ldots, \beta n_k'$.  This implies $S$ and $S'$ have identical Kunz posets, thereby completing the proof.  
\end{proof}

\begin{cor}\label{c:monoscopicfaces}
If $S = \<m, n_2, \ldots, n_k\>$ lies in the face $F \subset \mathcal C(\ZZ_{m})$ and $T = \<\alpha\> + \beta S$ is a monoscopic gluing lying in the face $F' \subset \mathcal C(\ZZ_{\beta m})$, then
$$\dim F' = \begin{cases}
\dim F & \text{if } \alpha \in \Ap(S;m); \\
\dim F + 1 & \text{if } \alpha \notin \Ap(S;m).
\end{cases}$$
\end{cor}




\begin{thebibliography}{HHHKR10}
\raggedbottom

\bibitem{alhajjarkunz}
E.~Alhajjar, R.~Russell, and M.~Steward, 
\emph{Numerical semigroups and Kunz polytopes}, 
Semigroup Forum 99 (2019), no.~1, 153--168.  

\bibitem{setoflengthsets}
J.~Amos, S.~Chapman, N.~Hine, J.~Paix\~ao, 
\emph{Sets of lengths do not characterize numerical monoids}, 
Integers \textbf{7} (2007), \#A50.

\bibitem{wilfmultiplicity}
W.~Bruns, P.~Garc\'ia-S\'anchez, C.~O'Neill, and D.~Wilburne,
\emph{Wilf's conjecture in fixed multiplicity}, 
preprint.  
Available at \textsf{arXiv:1903.04342}.

\bibitem{supersymmetric}
L.~Bryant and J.~Hamblin,
\emph{The maximal denumerant of a numerical semigroup},
Semigroup Forum 86 (2013), no.~3, 571--582.

\bibitem{wilfsurvey}
M.~Delgado, 
\emph{Conjecture of Wilf:\ a survey},
preprint.  Available at \texttt{arXiv:1902.03461}. 

\bibitem{oversemigroupcone}
M.~Hellus and R.~Waldi,
\emph{On the number of numerical semigroups containing two coprime integers p and q},
Semigroup Forum 90 (2015), no.~3, 833--842.

\bibitem{kaplanwilfconj}
N.~Kaplan,
\emph{Counting numerical semigroups by genus and some cases of a question of Wilf}
J.~Pure Appl.~Algebra 216 (2012), no.~5, 1016--1032.

\bibitem{kunzfaces1}
N.~Kaplan and C.~O'Neill, 
\emph{Numerical semigroups and polyhedra faces I:\ posets and the group cone},
preprint.  
Available at \textsf{arXiv:1912.03741}.

\bibitem{compseqs}
C.~Kiers, C.~O'Neill and V.~Ponomarenko,
\emph{Numerical semigroups on compound sequences},
Communications in Algebra \textbf{44} (2016), no.~9, 3842--3852.  

\bibitem{telescopic}
C.~Kirfel and R.~Pellikaan,
\emph{The minimum distance of codes in an array coming from telescopic semigroups},
IEEE Trans.\ Inform.\ Theory \textbf{41} (1995), no.~6, part~1, 1720--1732.

\bibitem{kunz}
E.~Kunz, 
\emph{\"Uber die Klassifikation numerischer Halbgruppen},
Regensburger Mathematische Schriften \textbf{11}, 1987.

\bibitem{kunznew}
E.~Kunz and R.~Waldi, 
\emph{Counting numerical semigroups},
preprint.  
Available at \textsf{arXiv:1410.7150}.

\bibitem{omidali}
M.~Omidali, 
\emph{The catenary and tame degree of numerical monoids generated by generalized arithmetic sequences}, 
Forum Mathematicum. Vol 24 (3), 627--640.

\bibitem{omidalirahmati}
M. Omidali and F. Rahmati, 
\emph{On the type and minimal presentation of certain numerical semigroups}, 
Communications in Algebra 37(4) (2009), 1275--1283.

\bibitem{numerical}
J.~Rosales and P.~Garc\'ia-S\'anchez, 
\emph{Numerical semigroups}, 
Developments in Mathematics, Vol. 20, Springer-Verlag, New York, 2009.

\bibitem{kunzcoords}
J.~Rosales, P.~Garc\'ia-S\'anchez, J.~Garc\'ia-Garc\'ia and M.~Branco, 
\emph{Systems of inequalities and numerical semigroups},
J.~Lond.\ Math.\ Soc.\ \textbf{65} (2002), no.~3, 611--623.

\bibitem{ziegler}
G.~Ziegler,
\emph{Lectures on polytopes}, 
Graduate Texts in Mathematics Vol.~152, Springer-Verlag, New York, 1995.


\end{thebibliography}
\end{document}